\documentclass[reqno]{amsart}

\usepackage{graphicx,mathtools,enumitem,xcolor}

\theoremstyle{plain}
\newtheorem{lemma}{Lemma}
\newtheorem{theorem}[lemma]{Theorem}
\newtheorem{corollary}[lemma]{Corollary}
\newtheorem{definition}[lemma]{Definition}

\theoremstyle{remark}
\newtheorem{remark}[lemma]{Remark}

\newcommand  {\G}{G}
\newcommand  {\K}{K}
\newcommand  {\N}{\mathbb{N}}
\newcommand  {\R}{\mathbb{R}}
\newcommand  {\I}{\mathbb{I}}

\renewcommand{\d}{\mathrm{d}}
\newcommand  {\e}{\mathrm{e}}

\newcommand  {\Ic}{\mathcal{I}}
\newcommand  {\Jc}{\mathcal{J}}

\newcommand{\erfc}{\operatorname{erfc}}

\newcommand{\eps}{\varepsilon}
\newcommand  {\HK}{\Theta}

\newcommand{\full}{{\operatorname{full}}}

\renewcommand{\star}{*}

\newcommand*{\norm}[3]{\lVert #2 \rVert_{#3}^{#1}}
\newcommand*{\abs}[1]{\lvert #1 \rvert}

\begin{document}
\title[Breakdown of precipitation bands]%
{Breakdown of Liesegang precipitation bands \\
in a simplified fast reaction limit\\ of the Keller--Rubinow model}

\author[Z. Darbenas]{Zymantas Darbenas}
\email[Z. Darbenas]{z.darbenas@jacobs-university.de}
\author[M. Oliver]{Marcel Oliver$^{*}$}
\thanks{$^{*}$Corresponding author, m.oliver@jacobs-university.de}
\email[M. Oliver]{m.oliver@jacobs-university.de}

\address[Z. Darbenas and M. Oliver]%
{School of Engineering and Science \\
 Jacobs University \\
 28759 Bremen \\
 Germany}


\date{\today}

\begin{abstract}
We study solutions to the integral equation
\begin{equation*}
  \omega(x) 
  = \Gamma - x^2 \int_{0}^1 
      K(\theta) \, H(\omega(x\theta)) \, \d \theta 
\end{equation*}
where $\Gamma>0$, $K$ is a weakly degenerate kernel satisfying, among
other properties, $K(\theta) \sim k \, (1-\theta)^\sigma$ as
$\theta \to 1$ for constants $k>0$ and $\sigma \in (0, \log_2 3 -1)$,
$H$ denotes the Heaviside function, and $x \in [0,\infty)$.  This
equation arises from a reaction-diffusion equation describing
Liesegang precipitation band patterns under certain simplifying
assumptions.  We argue that the integral equation is an analytically
tractable paradigm for the clustering of precipitation rings observed
in the full model.  This problem is nontrivial as the right hand side
fails a Lipschitz condition so that classical contraction mapping
arguments do not apply.

Our results are the following.  Solutions to the integral equation,
which initially feature a sequence of relatively open intervals on
which $\omega$ is positive (``rings'') or negative (``gaps'') break
down beyond a finite interval $[0,x^*]$ in one of two possible ways.
Either the sequence of rings accumulates at $x^*$ (``non-degenerate
breakdown'') or the solution cannot be continued past one of its
zeroes at all (``degenerate breakdown'').  Moreover, we show that
degenerate breakdown is possible within the class of kernels
considered.  Finally, we prove existence of generalized solutions
which extend the integral equation past the point of breakdown.
\end{abstract}

\maketitle

\section{Introduction}

Reaction-diffusion equations with discontinuous hysteresis occur in a
range of modeling problems \cite{BrokateS:1996:HysteresisPT,
KrasnoselskiiP:1989:SystemsH, Mayergoyz:1991:MathematicalMH,
Visintin:1994:DifferentialMH, Visintin:2014:TenIH}.  We are
particularly interested in non-ideal relays---two-valued operators
where the output switches from the ``off-state'' $0$ to the
``on-state'' $1$ when the input crosses a threshold $\beta$, and
switches back to zero only when the input drops below a lower
threshold $\alpha<\beta$.  There are different choices to define the
behavior of the relay at the threshold.  The relay may be restricted
to binary values and jump when the threshold is reached or exceeded.
Alternatively, the relay may be \emph{completed}: when the threshold
is reached but not exceeded, the relay may take fractional values
which can change monotonically in time; when the input drops below the
threshold without having crossed, the attained fractional value gets
``frozen in''.  See, e.g., \cite{CurranGT:2016:RecentAR} for a
detailed description of different relay behaviors.

Rigorous mathematical results are of two types.  For
reaction-diffusion equations with completed relays, weak limit
arguments lead to existence of solution
\cite{Visintin:1986:EvolutionPH, AikiK:2008:MathematicalMB} but not
necessarily their uniqueness and continuous dependence on the data.
For reaction-diffusion equations with non-completed non-ideal relays,
local well-posedness, including uniqueness and continuous dependence,
holds true provided that a certain transversality condition on the
data is satisfied.  The solution can be continued in time for as long
as the transversality condition remains satisfied
\cite{GurevichT:2012:UniquenessTS, GurevichST:2013:ReactionDE,
CurranGT:2016:RecentAR}.  We finally remark that for some types of
spatially distributed hysteresis, variational approaches may be
available \cite{MielkeTL:2002:VariationalFR}.

In this paper, we study an explicit example of a reaction-diffusion
equation with relay hysteresis which demonstrates that, in general,
global-in-time solutions require the notion of a completed relay.
Our example is motivated from the study of the fast reaction limit,
introduced by Hilhorst \emph{et al.}\ \cite{HilhorstHM:2007:FastRL,
HilhorstHM:2009:MathematicalSO}, of the Keller and Rubinow model for
Liesegang precipitation rings \cite{KellerR:1981:RecurrentPL}.  This
limit model, which we will refer to as the \emph{HHMO-model}, is a
scalar reaction-diffusion equation driven by a point source which is
constant in parabolic similarity variables with a reaction term
modeled by a relay with a positive upper threshold and zero lower
threshold.  As a consequence, at a fixed location in space, the
reaction, once switched on, can never switch off.  The loci of
reaction then form a spatial precipitation pattern.

Simple as it seems, an analysis of the HHMO-model faces the same type
of difficulty as the analysis of other reaction-diffusion equations
with relay hysteresis; in particular, the questions of global
uniqueness and continuous dependence on the data remain open.  Our aim
here is to provide insight into the essential features of the
distributed relay dynamics.  We make use of a remarkable feature of
the HHMO-model: it can be formally simplified to an equation,
different but qualitatively similar to the actual HHMO-model, that is
self-similar in parabolic similarity variables.  This new model, which
we shall refer to as the \emph{simplified HHMO-model}, reduces to a
single scalar integral equation, i.e., can be considered as a scalar
dynamical system with memory.  The simplified model is finally simple
enough that a fairly complete explicit analysis is possible, which is
the main contribution of this paper.

We prove that the binary precipitation pattern in the dynamics of the
simplified HHMO-model must break down in finite space-time.  Beyond
the point of breakdown, it can only be continued as a generalized
solution.  We think of the behavior prior to breakdown as analogous to
the well-posedness result for binary switching relays in the spirit of
Gurevich \emph{et al.}\ \cite{GurevichST:2013:ReactionDE} and the
behavior past the point of breakdown as generalized solutions in the
sense of Visintin \cite{Visintin:1986:EvolutionPH}.  While these
analogies are tentative and we make no claim that the simplified
HHMO-model reflects the behavior of true Liesegang precipitation
patterns, the study of this model offers a paradigm for the breakdown
of binary patterns.  In particular, it gives insight that breakdown
can happen in two distinct ways.  We believe that more general
models---which may not share the symmetry which makes the explicit
results of this paper possible---are capable of exhibiting the
behaviors observed here, so that the results of this paper provide a
lower bound on the complexity which must be addressed when studying
more general situations.  We also offer a possible perspective for a
reformulation of the problem that may lead to well-posedness past the
point of breakdown.

To be specific, the simplified HHMO-model can be formulated as
\begin{equation}
  \label{omega.explicit.0}
  \omega(x) 
  = \Gamma - x^2 \int_{0}^1 
      K(\theta) \, H(\omega(x\theta)) \, \d \theta \,,
\end{equation}
where $\omega(x)$ is the excess reactant concentration at the source
point, $\Gamma$ is a positive constant, $H$ denotes the Heaviside
function, and $K$ is a unimodal kernel, continuous on $[0,1]$,
continuously differentiable on $[0,1)$, and twice continuously
differentiable in the interior of this interval, with the following
properties:
\begin{enumerate}[label={\upshape(\roman*)}]
\item \label{i.k1} 
$K(\theta)$ is non-negative with $K(0)=K'(0)=0$,
\item \label{i.k2} 
$K(\theta) \sim k \, \sqrt{1-\theta}$ as
$\theta\to1$ for some $k>0$,
\item \label{i.k3} 
there exists $\theta^\star \in (0,1)$ such that $K''(\theta)>0$ for
$\theta \in (0,\theta^\star)$ and $K''(\theta)<0$ for
$\theta \in (\theta^\star,1)$.
\end{enumerate}
These properties imply, in particular, that $K>0$ on $(0,1)$ and
$K(1)=0$.

Clearly, at $x_0=0$, $\omega(x_0) = \Gamma >0$ and there must be a
point $x_1$ at which $\omega$ changes sign, i.e., where the
concentration falls below the super-saturation threshold.  Continuing,
we may define a sequence $x_i$ of loci where $\omega$ changes sign, so
that $(x_i,x_{i+1})$ corresponds to a ``ring'' or ``band'' where
precipitation occurs when $i$ is even and to a precipitation gap when
$i$ is odd.  Given the physical background of the problem, we might
think that the $x_i$ form an unbounded sequence, indicating that the
entire domain is covered by a pattern of rings or gaps, or, if the
sequence is finite, that the last ring or gap extends to infinity.

Our first result proves that this is not the case: The sequence $x_i$
either has a finite accumulation point $x^*$ or there is a finite
index $i$ such that $\omega$ cannot be extended past $x^*=x_i$ in the
sense of equation \eqref{omega.explicit.0}.  We call the former case
non-degenerate, the latter degenerate.

Our second result demonstrates the existence of degenerate solutions
to \eqref{omega.explicit.0}.  To this end, we present the construction
of a kernel where the solution cannot be continued past the first gap,
i.e., where the point of breakdown is $x^*=x_2$.

To extend the solution past $x^*$, we introduce the concept of
\emph{extended solutions}, reflecting the concept of a completed relay
in the spirit of \cite{Visintin:1986:EvolutionPH} and also
\cite{HilhorstHM:2009:MathematicalSO}.  Extended solutions are pairs
$(\omega,\rho)$ where $\omega \in C([0,\infty))$ and
\begin{equation}
  \label{HHMO.extended.0}
  \omega(x) = \Gamma-x^2 \int_0^1K(\theta) \,
    \rho(x\theta) \, \d\theta \,,
\end{equation}
subject to the condition that $\rho$ takes values from the Heaviside
graph, i.e.,
\begin{equation}
\label{Heaviside}
  \rho(y)\in H(\omega(y))
  = \begin{cases}
      0 & \text{if }\omega(y)<0 \,, \\
      [0,1] & \text{if }\omega(y)=0 \,, \\
      1 & \text{otherwise} \,.
    \end{cases}
\end{equation}
As our third result, we prove existence of extended solutions.
Extended solutions are unique under the condition that they are
\emph{regularly extended}, namely that $\omega$ remains identically
zero on some right neighborhood $[x^\star, b)$ past the point of
breakdown.

The remainder of the paper is structured as follows.  In
Section~\ref{s.HHMO}, we recall some background on Liesegang rings and
the fast reaction limit of the Keller--Rubinow model.  In
Section~\ref{s.simplified}, we simplify the model to the scalar
integral equation \eqref{omega.explicit.0} and present arguments and
numerical evidence that the simplified model reflects the qualitative
behavior of the full model.  We then proceed to show, in
Section~\ref{s.nondeg}, that the sequence of precipitation bands
either terminates finitely or has a finite accumulation point.  In
Section~\ref{s.degenerate}, we provide a construction that shows that
within the class of kernels considered, finite termination is
possible.  Section~\ref{exist.simplified} discusses extended solutions
in the sense of \eqref{HHMO.extended.0}.  We conclude with a brief
discussion and outlook.

\section{The Keller--Rubinow model in the fast reaction limit}
\label{s.HHMO}

Liesegang precipitation bands are structured patterns in
reaction-diffusion kinetics which emerge when, in a chain of two
chemical reactions, the second reaction is triggered upon exceeding a
supersaturation threshold and is maintained until the reactant
concentration falls below a lower so-called saturation threshold.
Within suitable parameter ranges, the second reaction will only ignite
in restricted spatial regions.  When the product of the final reaction
precipitates, these regions may be visible as ``Liesegang rings'' or
``Liesegang bands'' in reference to German chemist Raphael Liesegang
who described this phenomenon in 1896.  For a review of the history
and chemistry of Liesegang patterns, see
\cite{Henisch:1988:CrystalsGL,Stern:1954:LiesegangP}. 

Keller and Rubinow \cite{KellerR:1981:RecurrentPL} gave a quantitative
model of Liesegang bands in terms of coupled reaction-diffusion
equations, see \cite{DuleyFM:2017:KellerRM,
DuleyFM:2019:RegularizationOS} for recent results and further
references.  We note that there is a competing description in terms of
competitive growth of precipitation germs \cite{Smith:1984:OstwaldST}
which will not play any role in the following; see, e.g.,
\cite{KrugB:1999:MorphologicalCL} for a comparative discussion.

Our starting point is the fast reaction limit of the Keller--Rubinow
model, where the first-stage reaction rate constant is taken to
infinity and one of the first-stage reactant is assumed to be
immobile.  Hilhorst \emph{et al.}\
\cite{HilhorstHM:2007:FastRL,HilhorstHM:2009:MathematicalSO} proved
that, in this limit, the first-stage reaction can be solved explicitly
and contributes a point source of reactant to the second-stage
process.  Thus, only one scalar reaction-diffusion equation for the
second-stage reactant concentration $u = u(x,t)$ remains.  Formulated
on the half-line, the fast reaction limit, which we shall refer to as
the \emph{full} HHMO-model, reads as follows:
\begin{subequations}
  \label{e.original}
\begin{gather}
  u_t = u_{xx} +
        \frac{\alpha \beta}{2 \sqrt t} \, \delta (x - \alpha \sqrt{t})
        - p[x,t;u] \, u \,,
  \label{e.original.a} \\
  u_x(0,t) = 0 \quad \text{for } t \geq 0 \,, \\
  u(x,0) = 0 \quad \text{for } x>0 \,, \label{e.original.c}
\end{gather}
where $\alpha$ and $\beta$ are positive constants and the
precipitation function $p[x,t;u]$ is constrained by
\begin{equation}
\label{e.hhmo-p-weak-alternative}
  p(x,t)\in
  \begin{cases}
     0&\text{ if }\sup_{s\in[0,t]}u(x,s)<u^* \,,\\
     [0,1]&\text{ if }\sup_{s\in[0,t]}u(x,s)=u^* \,,\\
     1 &\text{ if }\sup_{s\in[0,t]}u(x,s)>u^* \,.
  \end{cases}
\end{equation}
\end{subequations}
In this expression, $u^*>0$ is the super-saturation threshold, i.e.,
the ignition threshold for the second-stage reaction.  For simplicity,
the saturation threshold is taken to be zero.  This means that once
the reaction is ignited at some spatial location $x$, it will not ever
be extinguished at $x$.

Hilhorst \emph{et al.}\ \cite{HilhorstHM:2009:MathematicalSO} proved
existence of weak solutions to \eqref{e.original}; the question of
uniqueness was left open.  It is important to note that a weak
solution is always a tuple $(u,p)$ where $p$ is constrained, but not
defined uniquely in terms of $u$, by
\eqref{e.hhmo-p-weak-alternative}.  The analytic difficulties lie in
the fact that the onset of precipitation is a free boundary in the
$(x,t)$-plane.  Moreover, the precipitation term is discontinuous, so
that most of the standard analytical tools are not applicable; in
particular, estimates based on energy stability fail.  In
\cite{Darbenas:2018:PhDThesis,DarbenasO:2018:UniquenessSK}, we are
able to prove uniqueness for at least an initial short interval of
time and derive a sufficient condition for uniqueness at later times;
we conjecture that it is possible to obtain instances of
non-uniqueness when the problem is considered with arbitrary smooth
initial data or smooth additional forcing.  One of the questions posed
in \cite{HilhorstHM:2009:MathematicalSO} is the problem of proving
that the precipitation function $p$ takes only binary values.

Numerical evidence suggests that after an initial transient period in
which a small number of rapidly shrinking rings is visible, the
solution appears to precipitate on single grid points whose exact
locations are unstable with respect to grid refinement.  Results
specifically for the HHMO-model are reported in
Section~\ref{s.simplified} below; Duley \emph{et al.}\
\cite{DuleyFM:2017:KellerRM} report similar behavior also for the
original Keller--Rubinow model.  The main result of this paper is that
we suggest a mechanism by which an actual breakdown of the ring
structure in the Keller--Rubinow model occurs.  It is made rigorous
for a simplified version of the HHMO-model, introduced in the
following, but displays features that are also seen in both of its
parent models.

\section{The simplified HHMO-model}
\label{s.simplified}

In the following, we detail the connection between the full HHMO-model
\eqref{e.original} and the integral equation \eqref{omega.explicit.0}.
The key observation is that, when written in a suitable equivalent
form, there are only two terms in the full model which do not possess
a parabolic scaling symmetry.  We cite a mixture of analytic and
numerical evidence that suggest that these terms have a negligible
impact on the long-time behavior of the solution: One of the neglected
terms represents linear damping toward equilibrium.  It is
asymptotically subdominant relative to the precipitation term;
moreover, its presence could only enhance relaxation to equilibrium.
The other term is observed to be asymptotically negligible as the
width of the precipitation rings decreases, hence its contribution
vanishes as the point of breakdown is approached.  Leaving only terms
which scale parabolically self-similarly, one of the variables of
integration in the Duhamel formula representation of the simplified
model can be integrated out, leaving an expression of the form
\eqref{omega.explicit.0} with a complicated, yet explicit expression
for the kernel $K$ which is shown, using a mixture of analysis and
numerical verification, to satisfy properties \ref{i.k1}--\ref{i.k3}.

\begin{figure}
\centering
\includegraphics[width=0.8\textwidth]{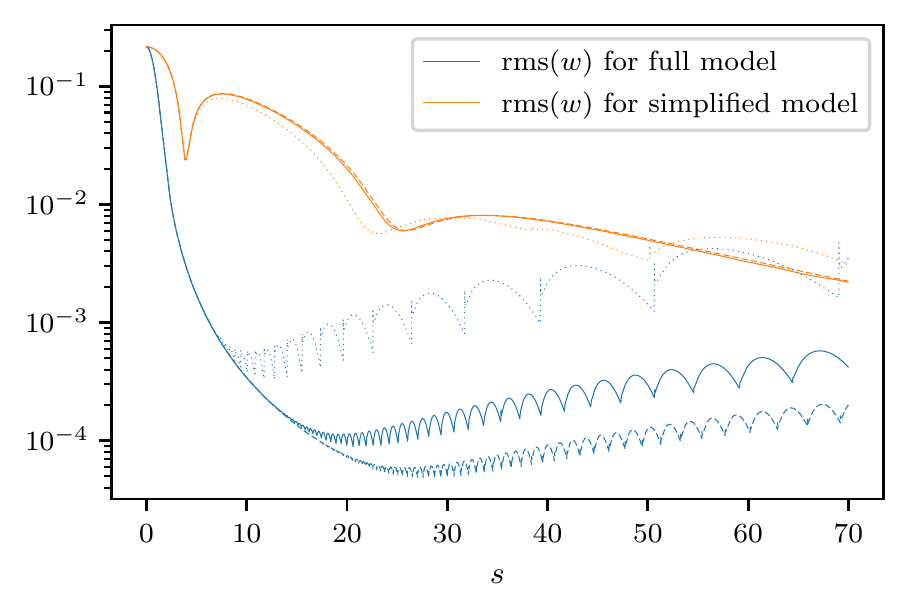}
\caption{Numerical verification of the convergence of the full and the
simplified HHMO-model to the self-similar profile $\Phi$.  Even though
the transients are different, the difference field $w$ converges to
zero in both cases.  Model parameters are $\alpha=\beta=1$ and
$u^*=0.2$.  The simulation shows that numerical artifacts at large
times decrease with improved resolution, where
$\Delta s = \Delta \eta = 10^{-2}$ (dotted lines),
$\Delta s = \Delta \eta = 10^{-3}$ (dashed lines), and
$\Delta s = \Delta \eta = 10^{-4}$ (solid lines).  See text for a
detailed discussion.}
\label{f.comparison}
\end{figure}

Numerical evidence suggests that solutions to the full HHMO-model
converge robustly to a steady state $\Phi(\eta)$ with respect to the
parabolic similarity variable $\eta = x/s$ as
$s = \sqrt{t} \to \infty$, see Figure~\ref{f.comparison} which is
explained in detail further below.  A proof of convergence to a steady
state is difficult for much the same reasons that well-posedness is
difficult, but \cite{DarbenasHO:2018:LongTA} were able to prove a
slightly weaker result: \emph{assuming} that the HHMO-solution
converges to a steady state $\Phi(\eta)$ at all, this steady state
must satisfy the differential equation
\begin{subequations}
  \label{e.v-selfsimilar}
\begin{gather}
  \Phi'' + \frac\eta2 \, \Phi' + \frac{\alpha\beta}2 \, \delta(\eta-\alpha)
  - \frac\gamma{\eta^2} \, H(\alpha-\eta) \, \Phi = 0 \,,
  \label{e.ode2} \\
  \Phi'(0) = 0 \,, \label{e.v-selfsimilar-b} \\
  \Phi(\eta) \to 0 \quad \text{as } \eta \to \infty \,,
  \label{e.v-selfsimilar-c} \\
  \Phi(\alpha) = u^* \,.
  \label{e.v-selfsimilar-d}
\end{gather}
\end{subequations}
In this formulation, $\gamma$ is an unknown constant.  To determine
$\gamma$ uniquely, this second order system has an additional internal
boundary condition \eqref{e.v-selfsimilar-d} which expresses that the
reactant concentration in the HHMO-model converge to the critical
value $u^*$ at the source point which, in similarity coordinates,
moves along the line $\eta=\alpha$.

There exists a unique solution $(\Phi,\gamma)$ to
\eqref{e.v-selfsimilar} with $\gamma>0$ and $\Phi$ given by
\begin{gather}
  \Phi(\eta)
  = \begin{dcases}
      \frac{u^*\,\eta^\kappa \,
            M \bigl(\frac\kappa2,\kappa+\frac12, -\frac{\eta^2}4 \bigr)}
           {\alpha^\kappa \, M \bigl(\frac\kappa2,\kappa+\frac12,
            -\frac{\alpha^2}4 \bigr)} & \text{ if }\eta<\alpha \,, \\
      \frac{u^*}{\erfc (\frac\alpha2)} \,
      \erfc \Bigl( \frac\eta2 \Bigr)
      & \text{ if }\eta\ge\alpha \,,
    \end{dcases}
  \label{e.phi.gamma}
\end{gather}
where $M$ is Kummer's confluent hypergeometric function
\cite{AbramowitzS:1972:HandbookMF}, $\kappa$ is a solution of the
algebraic equation 
\begin{equation}
  \label{eq.kappa.2}
  u^* = u^*_\gamma \equiv
  \left(\frac{\kappa \,
    M \bigl( \frac\kappa2+1,\kappa+\frac12, -\frac{\alpha^2}4 \bigr)}%
   {\alpha \,
    M \bigl( \frac\kappa2,\kappa+\frac12,-\frac{\alpha^2}4 \bigr)}
  + \, \frac{\exp \bigl(-\frac{\alpha^2}4 \bigr)}%
                {\sqrt\pi \, \erfc (\frac\alpha2 )}\right)^{-1}
  \frac{\alpha\beta}2 \,,
\end{equation}
and $\gamma = \kappa (\kappa-1)$, subject to the solvability
condition
\begin{equation}
  \label{e.solvability}
  u^* < u_0^* \,.
\end{equation}
In $x$-$t$ coordinates, the self-similar solution to
\eqref{e.v-selfsimilar} takes the form
\begin{equation}
  \label{e.phi.gamma.x-t}
  \phi(x,t)=\Phi(x/{\sqrt t}) \,.
\end{equation}
Throughout this paper, we assume that $\alpha$, $\beta$, and $u^*$
satisfy the solvability condition \eqref{e.solvability}, so that the
self-similar solution $\phi$ exists.  (For triples $\alpha$, $\beta$,
and $u^*$ which violate the solvability condition, the corresponding
weak solution precipitates in only a bounded region and the asymptotic
state is easy to determine explicitly; for details, see
\cite{DarbenasHO:2018:LongTA}.)

\begin{figure}
\centering
\includegraphics[width=\textwidth]{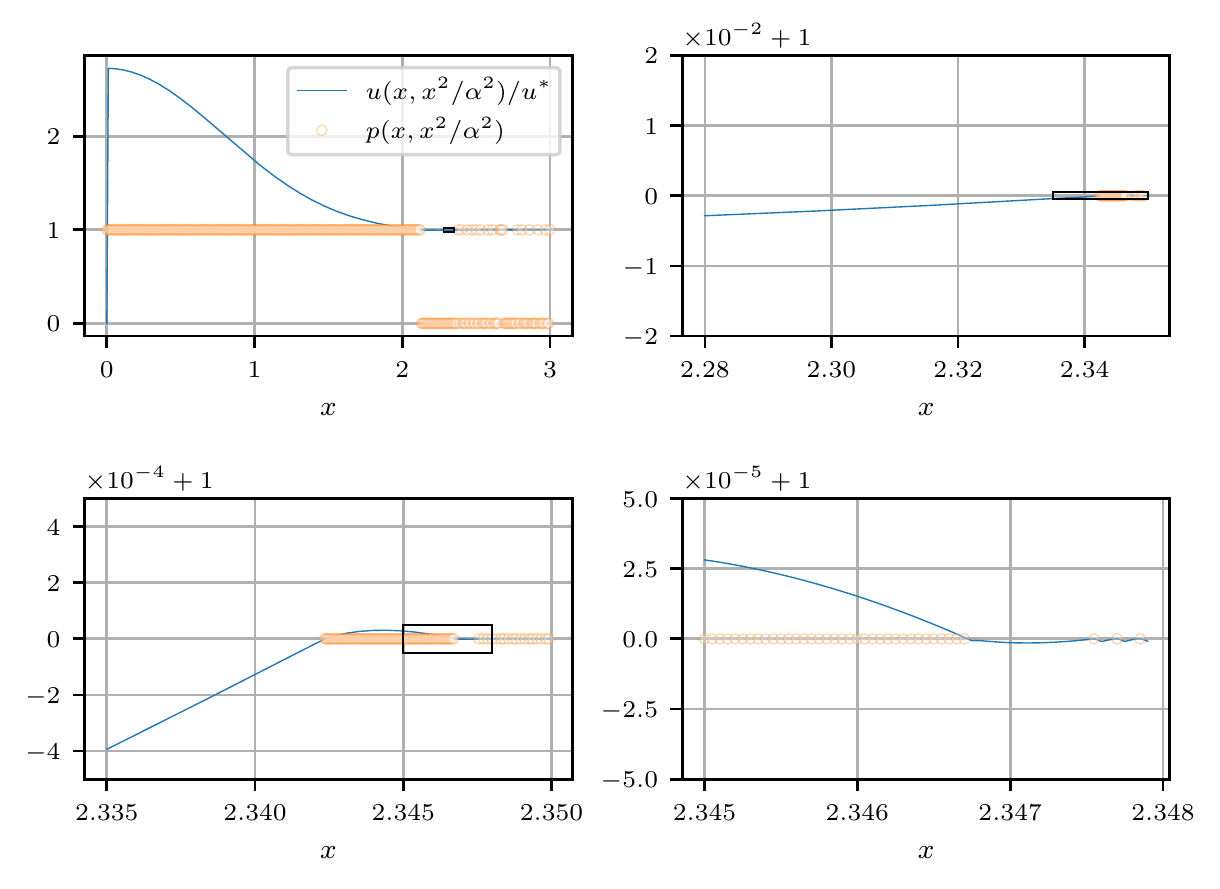}
\caption{Relative concentration $u/u^*$ and precipitation function $p$
along the parabola $t = x^2/\alpha^2$ for the full HHMO-model.  Each
subsequent graph zooms into the boxed area of the previous.  The
simulation parameters are $\alpha=\beta=1$, $u^*=0.2$,
$\Delta s = \Delta \eta = 5 \cdot 10^{-5}$.  Note that the $x$-scale
here coincides directly with similarity time $s=x/\alpha$ used in
Figure~\ref{f.comparison}.}
\label{f.omega-full}
\end{figure}

\begin{figure}
\centering
\includegraphics[width=\textwidth]{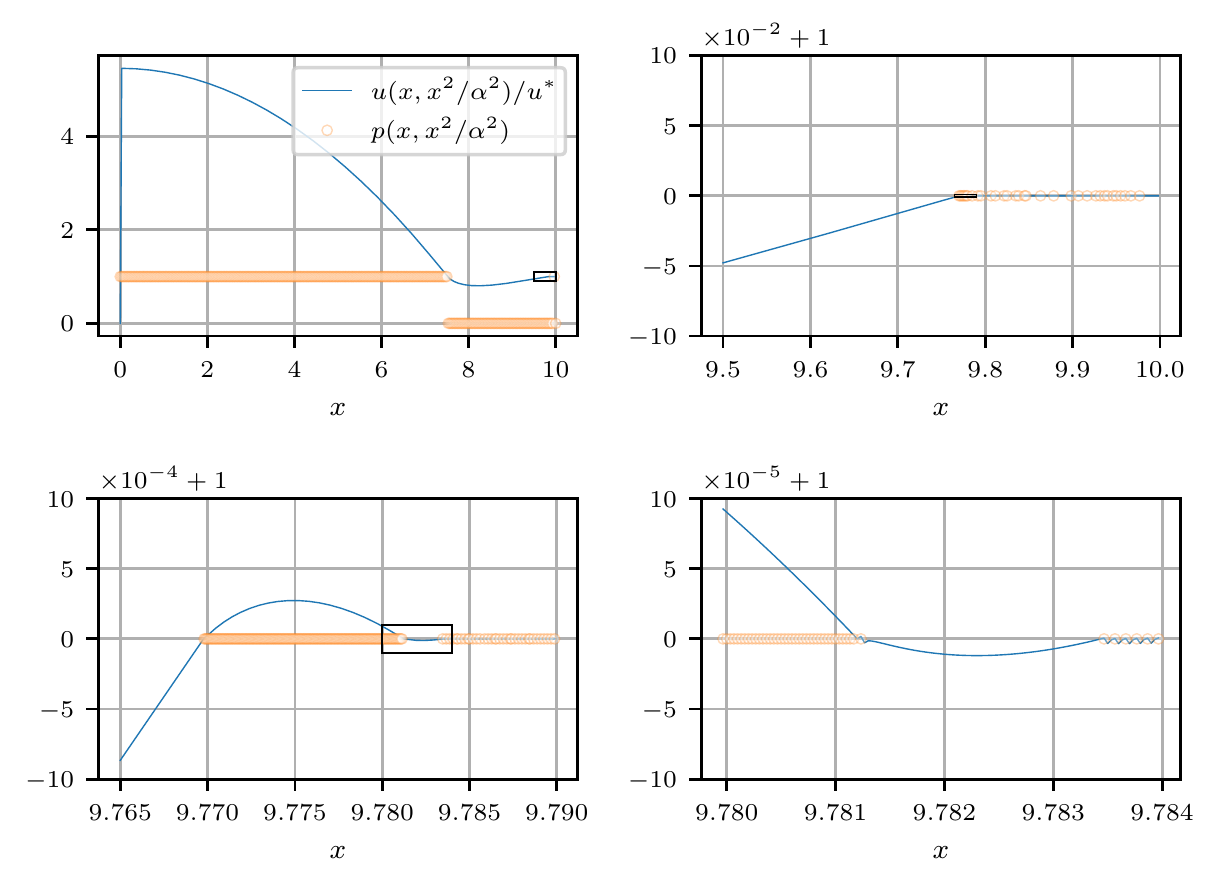}
\caption{Relative concentration $u/u^*$ and precipitation function $p$
along the parabola $t = x^2/\alpha^2$ for the simplified HHMO-model.
Each subsequent graph zooms into the boxed area of the previous.  The
simulation parameters are $\alpha=\beta=1$, $u^*=0.1$,
$\Delta s = \Delta \eta = 3.33 \cdot 10^{-5}$.  Note that the
$x$-scale here coincides directly with similarity time $s=x/\alpha$
used in Figure~\ref{f.comparison}.}
\label{f.omega-simplified}
\end{figure}

We now write $w=u-\phi$ to denote the difference between the solution
of the full HHMO-model \eqref{e.original} and the self-similar profile
\eqref{e.phi.gamma.x-t}.  Then $w$ solves the equation
\begin{subequations}
  \label{e.u-phi}
\begin{gather}
  w_t - w_{xx} + p \, w
  = \biggl( 
      \frac\gamma{x^2} \, H \Bigl( \alpha - \frac{x}{\sqrt t} \Bigr) 
      - p 
    \biggr) \, \phi \Bigl( \frac{x}{\sqrt t} \Bigr) \,, 
  \label{e.u-phi.a} \\
  w_x(0,t) = 0 \quad \text{for } t \geq 0 \,, \\
  w(x,0) = 0 \quad \text{for } x > 0 \,.
\end{gather}
\end{subequations}
To pass to the simplified HHMO-model, we make two changes to this
equation:
\begin{enumerate}[label={\upshape(\alph*)}]
\item\label{i.1} Precipitation is triggered on the condition
that $u>u^*$ on the line $x^2 = \alpha^2 t$, and
\item\label{i.2} the damping term $pw$ in \eqref{e.u-phi.a} is
neglected.
\end{enumerate}
The simplified model then reads
\begin{subequations}
  \label{e.simplified}
\begin{gather}
  w_t - w_{xx} 
  = \biggl( 
      \frac\gamma{x^2} - 
      H \Bigl(w\Bigl(x,\frac{x^2}{\alpha^2} \Bigr) \Bigr)
    \biggr) \, 
    H \Bigl( \alpha - \frac{x}{\sqrt t} \Bigr) \, \phi (x,t) \,, \\
  w_x(0,t) = 0 \quad \text{for } t \geq 0 \,, \\
  w(x,0) = 0 \quad \text{for } x > 0 \,.
\end{gather}
\end{subequations}

To illustrate the impact of these simplifications, we resort to
numerical simulation.  In the context of this problem, two general
comments about numerical simulation are necessary.

First, we are looking at the long-time behavior of the solution where
breakdown has already occurred early in the simulation.  For lack of a
clear alternative, numerical schemes are constructed under the
assumption of a binary precipitation function.  Thus, we cannot expect
pointwise convergence of the precipitation function; convergence of
the numerical precipitation function to the precipitation function of
a weak solution in the sense of \cite{HilhorstHM:2009:MathematicalSO}
can at best take place in the weak-$*$ topology.  The concentration,
on the other hand, may converge point-wise.  High-order convergence
cannot be expected in this setting so that we limit ourselves to the
lowest order finite difference approximations.

Second, the asymptotic profile is stationary in \emph{similarity}
variables, but a spatial grid cell of constant size in similarity
variables maps to a cell in physical space whose size increases in
time.  This leaves fundamentally two options:

\begin{enumerate}[label={\upshape(\roman*)}]
\item Compute on a fixed grid in physical space and accept that the
essential part of the solution will run out of the computational
domain in a finite time.

\item Compute on a fixed grid in similarity variables and accept that
the smallest unit of precipitation that can be represented by the
scheme will get coarser as time goes on, leading to spurious
oscillations with increasing amplitude at late times.
\end{enumerate}

It is conceivable that an $h$-$p$-adaptive method on an essentially
infinite domain in physical space could break this dichotomy, but
would also raise the issue of how much such a complex scheme can be
trusted, and how to adapt it to the essentially non-smooth nature of
this problem.  For these reasons, we choose not to take this route but
rather accept that accuracy is lost as time progresses.  The time of
validity can be extended by choosing a finer spatial grid in
similarity variables, or choosing a bigger domain in physical space.
For any given grid, however, the time of validity is finite.

For our code, we have chosen a finite difference scheme formulated in
similarity variables $\eta = x/\sqrt{t}$ and $s = \sqrt t$; it is
detailed in Appendix~\ref{a.numerics}.  The advantage is that it makes
the identification of the limit profile particularly easy; the
drawback is that it requires a transport scheme for the precipitation
function, making the code just slightly longer than a direct solve in
physical variables.

Figure~\ref{f.comparison} demonstrates that both the full HHMO-model
and the simplified model converge to the asymptotic profile $\Phi$,
i.e., $\lVert w(t) \rVert \to 0$ as $t \to \infty$.  As expected, due
to the lack of the linear damping term $pw$, the simplified model
takes longer to equilibrize, but the asymptotics remain unchanged.  We
note that the loss of accuracy with time, described above, is clearly
visible.  Increasing the numerical resolution moves the point of
visible onset of numerical error to larger times, but the behavior of
the scheme is fundamentally non-uniform in time.

Figures~\ref{f.omega-full} and~\ref{f.omega-simplified} show details
of the initial transient of the full and the simplified model,
respectively.  We see that even though the transients are
quantitatively different, the two models have the same qualitative
features: The amplitude of the variation of concentration about the
threshold concentration at the source point decreases extremely
rapidly, as does the width of the precipitation rings and gaps.  In
both simulations, we were able to clearly resolve two precipitation
rings and two gaps, where the last gap is only visible by zooming in
about five orders of magnitude.  We cannot determine whether there is
a third distinct ring; simulating this numerically would require at
least one order of magnitude more resolution in space, due to the
additional timestepping at least two orders of magnitude more in
computational expense.  In the following, we prove, for the simplified
model, that the ring structure must break down within a finite
interval; the simulations suggest that this interval is not
particularly large.

We have also observed that most of the quantitative change comes from
simplification \ref{i.2}.  Implementing simplification \ref{i.1}
without simplification \ref{i.2} yields a solution that is visually
indistinguishable from the solution to the full model.

Note that simplification \ref{i.1} implies that there is no
precipitation below the line $x^2 = \alpha^2 t$, even when $u>u^*$.
The advantage of this simplification is that onset of precipitation
now ceases to be a free boundary problem and follows parabolic
scaling.  A motivation for the validity of this simplification comes
from the following fact: it is proved in \cite{DarbenasHO:2018:LongTA}
that \emph{if} the solution to the full HHMO-model converges to a
parabolically self-similar profile as $t \to \infty$, then the
contribution to the HHMO-dynamics from precipitation below the
parabola $\alpha^2 \, t=x^2$ is asymptotically negligible.

Simplification \ref{i.2} is justified by the numerical observation
that the equation without the damping term $pw$ already converges to
the same profile, so that an additional linear damping toward the
equilibrium will not make a qualitative difference.  Moreover,
assuming that the HHMO-solution converges to equilibrium, $pw$ becomes
asymptotically small while the right hand side of \eqref{e.u-phi.a}
remains an order-one quantity.  It is very difficult, however, to
estimate the quantitative effect of \ref{i.1} and \ref{i.2} due to the
discontinuous reaction term and the free boundary of onset of
precipitation, so that a rigorous justification of these two steps
remains open.

To proceed, we extend the simplified HHMO-model \eqref{e.simplified} to
the entire real line by even reflection and abbreviate
\begin{equation}
  \rho(x) = \frac\gamma{x^2} - 
      H \Bigl(w\Bigl(x,\frac{x^2}{\alpha^2} \Bigr) \Bigr) \,.
\end{equation}
Proceeding formally, we apply the Duhamel principle---a detailed
justification of the Duhamel principle in context of weak solutions is
given in \cite{Darbenas:2018:PhDThesis}---then change the order of
integration and implement the change of variables $s=y^2/\zeta^2$, so
that
\begin{align}
  w(x,t) 
  & = \int_{-\alpha \sqrt t}^{\alpha \sqrt t} \int_{y^2/\alpha^2}^t 
        \HK(x-y,t-s) \, \phi (y,s) \, 
        \d s \, \rho(y) \, \d y 
      \notag \\
  & = 2 \int_{-\alpha \sqrt t}^{\alpha \sqrt t} 
      \int_{\lvert y \rvert/\sqrt t}^{\alpha} 
        \HK(x-y,t-y^2/\zeta^2) \, \frac{\Phi (\zeta)}{\zeta^3} \, 
        \d \zeta \, \rho(y) \, y^2 \, \d y 
  \label{e.fixed-point2}
\end{align}
where $\HK$ is the standard heat kernel
\begin{equation}
  \HK(x,t) = \begin{dcases}\frac1{\sqrt{4\pi t}} \,
                 \e^{-\tfrac{x^2}{4t}}&\text{if }t>0 \,,\\
		 0&\text{if }t\le0 \,.
	      \end{dcases}
\end{equation}

We are specifically interested in the solution on the parabola
$x^2=\alpha^2 \, t$.  For notational convenience, we assume in the
following that $x$ is nonnegative; solutions for negative $x$ are
obtained by even reflection.  Then, setting
$\omega(x) = w(x,x^2/\alpha^2)$ and inserting the fundamental solution
of the heat equation explicitly, we find
\begin{align}
  \omega(x) 
  & = \frac1{\sqrt \pi} \int_{-x}^{x} 
        \int_{\alpha \lvert y \rvert/x}^{\alpha}
        \frac1{\sqrt{\tfrac{x^2}{\alpha^2}-\tfrac{y^2}{\zeta^2}}} \,
        \exp \biggl(
               - \frac{(x-y)^2}%
                      {4 \, \bigl(
                              \tfrac{x^2}{\alpha^2}-\tfrac{y^2}{\zeta^2}
                            \bigr)}
             \biggr) \,
        \frac{\Phi (\zeta)}{\zeta^3} \, \d \zeta \, 
        \rho(y) \, y^2 \, \d y 
      \notag \\
  & = \frac1x\int_{-x}^{x} G \Bigl( \frac{y}x \Bigr) \, \rho(y) \, y^2 \, \d y  
      \notag \\
  & = x^2 
      \int_{-1}^1 G(\theta) \, \rho(x\theta) \, \theta^2 \, \d \theta \,,
  \label{e.omega1}
\end{align}
where
\begin{equation}
  \label{e.G}
  G(\theta) 
  = \frac\alpha{\sqrt \pi} \int_{\alpha\lvert\theta\rvert}^\alpha
    \frac1{\sqrt{\zeta^2-\alpha^2 \, \theta^2}} \,
        \exp \biggl(
               - \frac{\zeta^2 \, \alpha^2 \, (1-\theta)^2}%
                      {4 \, (\zeta^2 - \alpha^2 \, \theta^2)}
             \biggr) \,
        \frac{\Phi (\zeta)}{\zeta^2} \, \d \zeta \,.
\end{equation}
Inserting the explicit expression for $\rho$ into \eqref{e.omega1} and
noting that $\omega$ is extended to negative arguments by even
reflection, we obtain
\begin{equation}
  \label{omega.explicit}
  \omega(x) 
  = \Gamma - x^2 \int_{0}^1 
      K(\theta) \, H(\omega(x\theta)) \, \d \theta
\end{equation}
with 
\begin{equation}
  \Gamma = \gamma \int_{-1}^1 G(\theta) \, \d \theta 
  \label{e.Gamma}
\end{equation}
and 
\begin{equation}
\label{e.K.G}
  K(\theta) = \theta^2 \, (G(\theta) + G(-\theta)) \,.
\end{equation}
A graph of the kernel $K$ is shown in Figure~\ref{f.K}.  In
Appendix~\ref{kernel.prop}, we give a combination of analytic and
numerical evidence showing that this kernel satisfies properties
\ref{i.k1}--\ref{i.k3} stated in the introduction.

\begin{figure}
\centering
\includegraphics[width=0.8\textwidth]{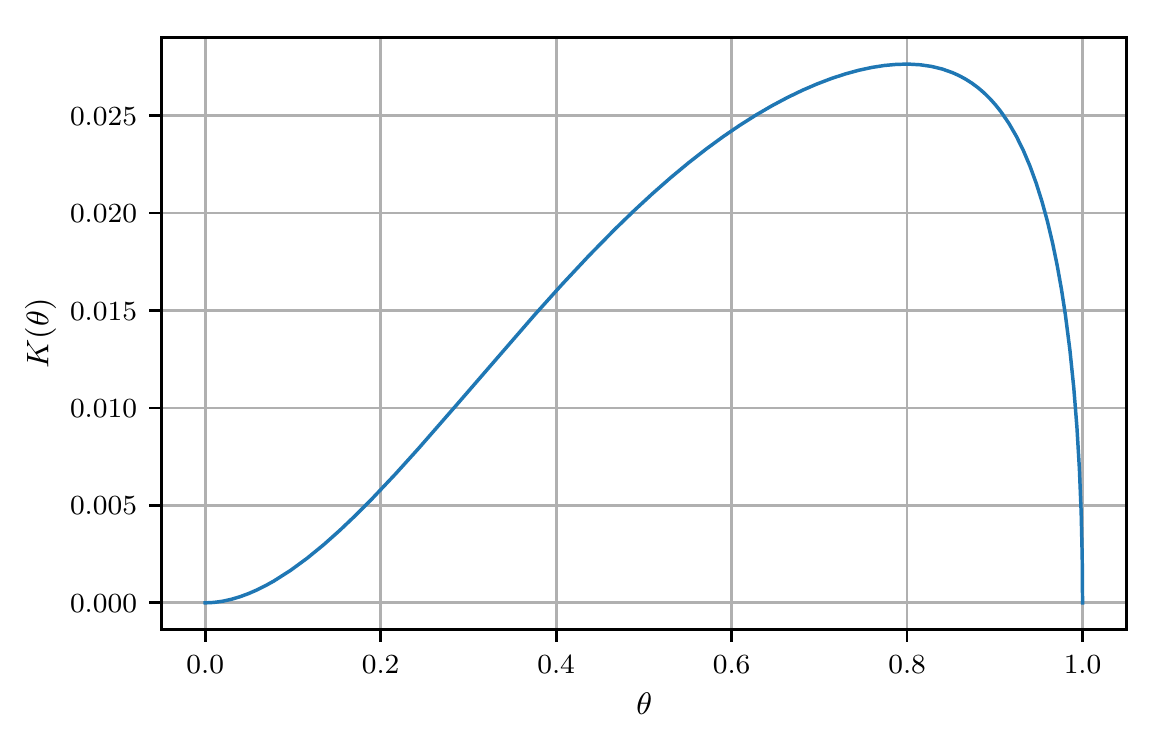}
\caption{Plot of the kernel $K(\theta)$ for $\alpha=\beta=1$ and
$u^*=0.15$.} 
\label{f.K}
\end{figure}

We conclude that the simplified HHMO-model implies an integral
equation of the form \eqref{omega.explicit.0}.  Vice versa, given a
solution $\omega$ to \eqref{omega.explicit}, we can reconstruct a
solution to the PDE-formulation of the simplified HHMO-model.  Indeed,
setting
\begin{align}
  W(x,t) 
  &= \int_0^t \int_{-\alpha \sqrt s}^{\alpha \sqrt s}
      \HK(x-y,t-s) \, \Bigl(\frac\gamma{y^2}-H(\omega(y))\Bigr) \, 
      \phi (y,s) \, \d y \, \d s \,,
  \label{e.reconstruction}
\end{align}
we can repeat the calculation leading to \eqref{e.omega1}, which
proves that $W(x,\alpha^{-2}x^2)=\omega(x)$.  Thus, $W$ solves
\eqref{e.fixed-point2} so that it provides a mild solution to
\eqref{e.simplified}.

\section{Non-degenerate breakdown of precipitation bands}
\label{s.nondeg}

In this section, we investigate the structure of solutions to the
integral equation \eqref{omega.explicit.0} for kernels $K$ which
satisfy assumptions \ref{i.k1}--\ref{i.k3}.  Specifically, we seek
solutions $\omega$ defined on a half-open interval $[0,x^*)$ or on
$[0,\infty)$ which change sign at isolated points $x_i$ for
$i=1,2,\dots$, ordered in increasing sequence.  Setting $x_0=0$ and
noting that precipitation must occur in a neighborhood of the origin
if it sets in at all, the precipitation bands are the intervals
$(x_i,x_{i+1})$ for even integers $i\geq 0$.  Hence,
\begin{equation}
  H(\omega(z)) 
  = \sum_{i \text{ even}} \I_{[x_i,x_{i+1}]}(z) \,,
\end{equation}
where we write $\I_A$ to denote the indicator function of a set $A$.
Thus, the one-dimensional precipitation equation
\eqref{omega.explicit.0} takes the form
\begin{align}
  \omega(x) 
  & = \Gamma 
      - x^2 \sum_{\substack{i \text{ even}\\ x_{i} < x}}
        \int_{x_i/x}^{\min\{x_{i+1}/x,1\}} K (\theta) \, \d \theta
      \notag \\
  & = \Gamma - \sum_{x_{i} < x} (-1)^i \, \rho_i(x) 
  \label{e.abstract-prec-eqn}
\end{align}
with
\begin{equation}
  \rho_i(x) =  x^2 \int_{x_i/x}^{1} K (\theta) \, \d \theta \,.
\end{equation}
For $x \geq x_{n-1}$, we also define the partial sums
\begin{equation}
  \omega_n(x) 
  = \Gamma - \sum_{i=0}^{n-1} (-1)^i \, \rho_i(x) \,.
  \label{e.partialsum}
\end{equation}
Thus, $\omega_n(x) = \omega(x)$ for $x \in [x_{n-1}, x_n]$.

With this notation in place, we are able to define the notion of
degenerate solutions.

\begin{definition} \label{d.degenerate}
A solution $\omega$ to \eqref{omega.explicit.0} is \emph{degenerate}
if \eqref{e.abstract-prec-eqn} holds up to some finite
$x_i \equiv x^*$ and it is not possible to apply this formula on
$[x^*,x^*+\eps)$ for any $\eps>0$; it is \emph{non-degenerate} if
$\omega$ possesses a finite or infinite sequence of isolated zeros
$\{x_i\}$ and the solution can be continued in the sense of
\eqref{e.abstract-prec-eqn} to some right neighborhood of any of its
zeros.
\end{definition}

In the remainder of this section, we characterize non-degenerate
solutions.  We cannot exclude that a solution is degenerate, i.e.,
that it cannot be continued at all beyond an isolated root; in fact,
Section~\ref{s.degenerate} shows that kernels with degenerate
solutions exist.  We note that a degenerate solution provides an
extreme scenario of a breakdown in which the solution reaches
equilibrium in finite time.  Thus, the main result of this section,
Theorem~\ref{th.shrink} below, can be understood as saying that even
when the solution is non-degenerate, it still fails to exist outside
of a bounded interval.

\begin{lemma} \label{l.infrings} 
Suppose $K \in C([0,1])$ is non-negative, strictly positive somewhere,
and $K(\theta) = o(\theta)$ as $\theta \to 0$.  Then a non-degenerate
solution to \eqref{omega.explicit.0} has an infinite number of
precipitation rings.
\end{lemma}

\begin{proof}
A non-degenerate solution, by definition, is a solution that can be
extended to the right in some neighborhood of any of its zeros.  Now
suppose there is a largest zero $x_{n}$.  Then $\omega$ is
well-defined and equals $\omega_{n+1}$ on $[x_n,\infty)$.  Now let
$x>x_n$ and consider the limit $x \to \infty$.  When $n$ is even,
since $\rho_{i+1}(x) < \rho_i(x)$,
\begin{equation}
  0 < \omega(x) < \Gamma - \rho_n(x) \to -\infty \,,
\end{equation}
a contradiction.  When $n$ is odd,
\begin{align}
  0 > \omega(x) 
  & > \Gamma - x^2 \int_0^{x_n/x} K(\theta) \, \d \theta 
      \notag \\
  & > \Gamma - x^2 \, 
      \sup_{\theta \in [0,x_{n}/x]} \frac{K(\theta)}{\theta}
      \int_0^{x_{n}/x} \theta \, \d\theta
      \notag \\
  & = \Gamma - \frac12 \, x_{n}^2 \, 
      \sup_{\theta\in[0,x_{n}/x]} \frac{K(\theta)}{\theta} 
      \to \Gamma > 0 \,,
\end{align}
once again a contradiction.
\end{proof}

\begin{lemma} \label{rings.schrinking}
Suppose $K \in C([0,1]) \cap C^1([0,1))$ with $K(1)=0$ and
$K'(\theta) \to -\infty$ as $\theta \to 1$.  Let $\omega$ be a
non-degenerate solution to \eqref{omega.explicit.0} with an infinite
number of precipitation rings.  Then $x_{2n}/x_{2n+1} \to 1$ as
$n\to\infty$.  Moreover, $x_{2n+1}-x_{2n}$, the width of the $n$th
precipitation ring, converges to zero.
\end{lemma}

\begin{proof}
When the sequence $\{x_i\}$ is bounded, the claim is obvious.  Thus,
assume that this sequence is unbounded.  Since $K'$ is negative on
$(1-\eps,1)$ for sufficiently small $\eps>0$, $K$ is positive on this
interval.  As in the proof of Lemma~\ref{l.infrings},
\begin{equation}
\label{r.eq}
  0 \equiv \omega(x_{2n+1}) 
  < \Gamma - \rho_{2n}(x_{2n+1}) 
  = \Gamma 
    - x_{2n+1}^2 \int_{x_{2n}/x_{2n+1}}^1 K(\theta) \, \d \theta \,.
\end{equation}
We can directly conclude that $x_{2n}/x_{2n+1} \to 1$ as
$n \to \infty$.  Further, noting that $K(1)=0$ and using the
fundamental theorem of calculus, we obtain
\begin{align}
  \Gamma 
  & > x_{2n+1}^2 
      \int_{x_{2n}/x_{2n+1}}^1 (K (\theta) - K(1)) \, \d \theta 
      \notag \\
  & = - x_{2n+1}^2 \int_{x_{2n}/x_{2n+1}}^1
        \int_\theta^1 K'(\zeta) \, \d \zeta \, \d \theta
      \notag \\
  & = - x_{2n+1}^2 \int_{x_{2n}/x_{2n+1}}^1 K'(\zeta) 
        \int_{x_{2n}/x_{2n+1}}^\zeta \d \theta \, \d \zeta 
      \notag \\
  & = - \frac12 \, K'(\zeta_n) \, (x_{2n+1}-x_{2n})^2   
\end{align}
where, by the mean value theorem of integration, the last equality
holds for some $\zeta_n \in [x_{2n}/x_{2n+1},1]$.  Since
$\zeta_n \to 1$, we conclude that $x_{2n+1}-x_{2n} \to0$ as
$n \to \infty$.
\end{proof}

\begin{theorem} \label{th.shrink}
Suppose $K \in C([0,1])$, differentiable with absolutely continuous
first derivative on $[0,z]$ for every $z\in(0,1)$, and unimodal, i.e.,
there exists $\theta^* \in (0,1)$ such that $K''(\theta)>0$ for
$\theta \in (0,\theta^*)$ and $K''(\theta)<0$ for
$\theta \in (\theta^*,1)$, and that
$K(\theta) \sim k \, (1-\theta)^\sigma$ for some $k>0$ and
$\sigma \in (0, \log_2 3 -1)$ as $\theta \to 1$. Further, assume that
equation \eqref{omega.explicit.0} has a non-degenerate solution
$\omega$ with an infinite number of precipitation rings. Then its
zeros have a finite accumulation point.
\end{theorem}

\begin{proof}
We begin by recalling the second order mean value theorem, which
states that for a twice continuously differentiable function $f$ and
nodes $a<b<c$ there exists $y \in [a,c]$ such that
\begin{equation}
  \frac{f(c)-f(b)}{c-b} - \frac{f(b)-f(a)}{b-a}
  = \frac{c-a}2 \, f''(y) \,.
  \label{f.second.der}
\end{equation}
We apply this result to the partial sum function $\omega_n$ with
$a=x_n$, $b=x_{n+1}$, and $c=x \in (x_{n+1},x_{n+2}]$.  We note that
$\omega_n(x_n)=0$.  Further, subtracting \eqref{e.partialsum} from
\eqref{e.abstract-prec-eqn}, we obtain, for $x \in [x_{n+1},x_{n+2}]$,
that
\begin{equation}
  \omega_n(x) = \omega(x) + (-1)^n \, (\rho_{n}(x) - \rho_{n+1}(x)) 
  \label{e.omegax}
\end{equation}
so that, in particular, $\omega_n(x_{n+1}) = (-1)^{n} \,
\rho_{n}(x_{n+1})$.  Equation \eqref{f.second.der} then reads
\begin{equation}
  \frac{\omega(x) + 
        (-1)^n \, (\rho_{n}(x) - \rho_{n+1}(x) - \rho_n(x_{n+1}))}%
       {x-x_{n+1}}
  - \frac{(-1)^{n} \, \rho_{n}(x_{n+1})}{x_{n+1}-x_n}
  = \frac{x-x_n}2 \, \omega_n''(y)
  \label{e.second.der2}
\end{equation}
for some $y \in [x_n,x]$.

To estimate
the right hand expression, we compute
\begin{equation}
  \omega_n''(y)
  = \sum_{i=0}^{n-1} (-1)^i \, F \Bigl(\frac{x_i}{y} \Bigr)
  \equiv \sum_{i=0}^{n-1} (-1)^i \, f_i 
\end{equation}
where
\begin{equation}
\label{def.F}
  F(z) = z^2 \, K'(z) - 2 \, z \, K(z) -
         2 \int^1_z K(\theta) \, \d\theta \,.
\end{equation}
By direct computation, $F'(z) = z^2 \, K''(z)$.  Since $K$ is
unimodal, this implies that $F$ has an isolated maximum on $[0,1]$.

Now suppose that $x \in (x_{n+1},x_{n+2})$.  We consider two separate
cases.  When $n$ is even, for every $y \geq x_{n-1}$ there exists a
unique odd index $\ell$ such that the sequence of $f_i$ is strictly
increasing for $i = 1, \dots, \ell-1$ and is strictly decreasing for
$i = \ell+1, \dots, n-1$.  Hence,
\begin{align}
  \omega_n''(y) 
  & = f_0 + (-f_1 + f_2) + \dots 
      - f_\ell + (f_{\ell+1} - f_{\ell+2}) + \dots 
      + (f_{n-2} - f_{n-1})
      \notag \\
  & > f_0 - f_\ell
      \geq F (0) - \max_{z \in [0,1]} F(z)
      \equiv -M \,,
\end{align}
where $M$ is a strictly positive constant.  Further, $\omega(x)<0$.
Inserting these two estimates into \eqref{e.second.der2}, we obtain
\begin{equation}
  \frac{\rho_{n}(x) - \rho_{n+1}(x) - \rho_n(x_{n+1})}%
       {x-x_{n+1}}
  - \frac{\rho_{n}(x_{n+1})}{x_{n+1}-x_n}
  > - \frac{x-x_n}2 \, M \,.
  \label{e.main-ineq}
\end{equation}

When $n$ is odd, for every $y \geq x_{n-1}$ there exists a unique even
index $\ell$ such that the sequence of $f_i$ is strictly increasing
for $i = 0, \dots, \ell-1$ and is strictly decreasing for $i = \ell+1,
\dots, n-1$.  Hence,
\begin{align}
  \omega_n''(y) 
  & = (f_0 - f_1) + \dots 
      + f_\ell + (-f_{\ell+1} + f_{\ell+2}) + \dots 
      + (- f_{n-2} + f_{n-1})
      \notag \\
  & < f_\ell
    \leq \max_{z \in [0,1]} F(z)
    = M + F (0) < M
\end{align}
Further, $\omega(x)>0$.  As before, inserting these two estimates into
\eqref{e.second.der2}, we obtain
\begin{equation}
  - \frac{\rho_{n}(x) - \rho_{n+1}(x) - \rho_n(x_{n+1})}%
       {x-x_{n+1}}
  + \frac{\rho_{n}(x_{n+1})}{x_{n+1}-x_n}
  < \frac{x-x_n}2 \, M \,.
  \label{e.main-ineq2}
\end{equation}
Thus, we again obtain an estimate of exactly the form
\eqref{e.main-ineq} and we do not need to further distinguish between
$n$ even or odd.

To proceed, we define
\begin{equation}
  R(\theta) = \dfrac{\int_\theta^1 K(\zeta) \, \d \zeta}%
                    {k \int_\theta^1 \left(1-\zeta\right)^\sigma \, \d \zeta}
\end{equation}
so that, by assumption, $R(\theta) \to 1$ as $\theta \to 1$.  Further,
\begin{equation}
  \rho_i(x) = \frac{k}{1+\sigma} \, x^2 \,
              \Bigl( 1 - \frac{x_i}x \Bigr)^{1+\sigma} \,
              R \Bigl( \frac{x_i}x \Bigr) \,.
\end{equation}
Changing variables to
\begin{equation}
  d_n = x_{n+1} - x_n \,, \quad
  r_n = \frac{x_n}{x_{n+1}} \,, \quad \text{and} \quad
  q = \frac{x-x_{n+1}}{d_n} \,,
\end{equation}
we write inequality \eqref{e.main-ineq} in the form
\begin{equation}
  G(q) 
  > - S_0 (r_n,q)
  + S_1(r_n,q) \, (1+q)^{1+\sigma}
  - S_2(r_n,q) \, q^{1+\sigma}
  - S_3(r_n,q) \, (q+1)
  \label{e.G-ineq}
\end{equation}
where
\begin{subequations}
\begin{gather}
  S_0 (r,q) 
  = -q \, (q+1) \, \biggl( \frac{1-r}{1+(1-r)q} \biggr)^{1-\sigma} \, 
    \frac{M \, (1+\sigma)}{2k} \,, \\
  S_1 (r,q) 
  = 1 - R \biggl( \frac{r}{1+(1-r)q} \biggr) \,, \\
  S_2 (r,q) 
  = 1 - R \biggl( \frac1{1+(1-r)q} \biggr) \,, \\
  S_3 (r,q) 
  = 1 - \biggl( \frac1{1+(1-r)q} \biggr)^{1-\sigma} \, R(r) \,,
\end{gather}
and
\begin{gather}
  G(q) = (1+q)^{1+\sigma} - q^{1+\sigma} - q - 1 \,.
  \label{e.Gdef}
\end{gather}
\end{subequations}
Observe that $G(0)=0$, $G'(0)>0$, $G''(q)<0$ for all $q>0$, and
$G(1)=2^{1+\sigma}-3<0$.  Hence, there is a unique root
$q^* \in (0,1)$ such that $G(q)<0$ for all $q>q^*$.  Now fix
$\varepsilon>0$, define
\begin{gather}
  q_n = \frac{x_{n+2}-x_{n+1}}{x_{n+1}-x_n} \,,
\end{gather}
and consider any even index $j$ for which $q_j > q^* + \varepsilon$.
Since \eqref{e.G-ineq} was derived under the assumption
$x \in (x_{j+1},x_{j+2})$, or equivalently $q\in(0,q_j)$, this
inequality must hold for each tuple $(r_j, q^* + \varepsilon)$.  Now
if there were an infinite set of indices for which
$q_j > q^* + \varepsilon$, we could pass to the limit $j \to \infty$
on the subsequence of such indices.  Since $K''(\theta)<0$ for
$\theta\in(\theta^*,1)$ and $K(\theta) \sim k \, (1-\theta)^\sigma$ as
$\theta \to 1$, Lemma~\ref{rings.schrinking} is applicable and implies
that $r_{2k}=x_{2k}/x_{2k+1}\to1$.  As for any fixed $q$, each of the
$S_i(r,q)$ converges to zero as $r\to 1$, we arrive at the
contradiction $G(q^* + \varepsilon)>0$.  Hence,
\begin{equation}
  \limsup_{\substack{k \to \infty\\ k \text{ even}}} q_k 
  \le q^* < 1 \,. 
  \label{e.limsup}
\end{equation}

To extend this result to odd $n$, we note that
\begin{gather}
  r_{n+1} 
  = \frac{x_{n+1}}{x_{n+2}}
  = \frac1{1+q_n(1-r_n)}
  = \frac{(1-r_n)(1-r_n q_n)}{1+q_n(1-r_n)} + r_n
  > r_n 
\end{gather}
for all large enough even $n$.  This implies an even stricter bound on
the right hand side of \eqref{e.G-ineq} when $n$ is replaced by $n+1$,
so that \eqref{e.limsup} holds on the subsequence of odd integers as
well.

Altogether, this proves that the sequence of internodal distances
$d_n = x_{n+1}-x_n$ is geometric, thus the $x_n$ have a finite limit.
\end{proof}

\begin{remark}
Note that in the proof of Theorem~\ref{th.shrink}, we only need a
result which is weaker than the statement of
Lemma~\ref{rings.schrinking}, namely that $r_n \to 1$ for even $n$
going to infinity
\end{remark}

\begin{remark}
Note that the argument yields an explicit upper bound for $q_n$,
namely
\begin{equation}
  \limsup_{n\to \infty} q_n \le q^* < 1 
\end{equation}
where, as in the proof, $q^*$ is the unique positive root of $G$,
which is defined in \eqref{e.Gdef}.
\end{remark}

\begin{remark}
\label{r.shrink-generalization}
It is possible to relax the unimodality condition in the statement of
Theorem~\ref{th.shrink}.  In fact, it suffices that
$\lim_{\theta\nearrow1}K''(\theta)=-\infty$.  Indeed, assume that
$K''$ is defined on $U_K$. Take any $\theta^\star\in(0,1)$ such that
$K''<0$ on $[\theta^\star,1)\cap U_K$.  Changing variables in
\eqref{e.partialsum}, we obtain
\begin{equation}
  \omega_n(x)
  = \Gamma-x\int_0^x \I_{[0,x_n]}(y) \,
    K \Bigl( \frac yx \Bigr) \, H(\omega(y)) \,\d y \,.
\end{equation}
When $n$ is even and $x\in(x_{n+1},x_{n+2})$, the singularity of $K'$
and $K''$ at $\theta=1$ is separated from the domain of integration.
We can therefore differentiate under the integral, so that
\begin{align}
  \omega_n'(x)
    = \frac1x\int_0^x \I_{[0,x_n]}(y) \, K' \Bigl( \frac yx \Bigr) \,
        y \, H(\omega(y)) \, \d y
      - \int_0^x\I_{[0,x_n]}(y) \, K \Bigl( \frac yx \Bigr) \,
        H(\omega(y)) \, \d y
\end{align}
and
\begin{align}
  \omega_n''(x)
  & = - \frac1{x^3}\int_0^x \, \I_{[0,x_n]}(y) \,
        K'' \Bigl( \frac yx \Bigr) \, y^2 \, H(\omega(y)) \, \d y
      \notag \\
  & = - \int_0^1 \I_{[0,x_n/x]}(\theta) \, K''(\theta) \, \theta^2 \, 
        H(\omega(\theta x)) \, \d \theta  
      \notag \\
  & \ge  -\int_0^{\theta^\star} \lvert K''(\theta) \rvert \,
        \theta^2 \, \d\theta >-\infty \,.
\end{align}

When $n$ is odd, then for every $x\in(x_{n+1},x_{x+2})$ there is an
even index $\ell$ such that $x_{\ell-1}<x\theta^\star\le x_{\ell+1}$
(with the provision that $x_{-1}=0$), and
\begin{equation}
  \omega_n''(x)
  = - x^{-3} \int_0^{x_{\ell-1}} K'' \Bigl( \frac yx \Bigr) \, y^2 \,
      H(\omega(y)) \, \d y
    + \sum_{i=\ell}^{n-1}(-1)^i \, F \Bigl( \frac{x_i}x \Bigr) \,,
\end{equation}
where $F$ is as in \eqref{def.F}.  As in the proof of the theorem,
$F'(z)=z^2 \, K''(z)<0$ on $[\theta^\star,1)$ so that
$M^*=\max_{z\in[0,1]}F(z)$ is finite and
\begin{equation}
  \sum_{i=\ell}^{n-1}(-1)^i \, F \Bigl(\frac{x_i}x \Bigr)
  \le F \Bigl( \frac{x_\ell}x \Bigr)
  \le M^* \,.
\end{equation}
Therefore,
\begin{align}
  \omega_n''(x)
  \le - \int_0^{x_\ell/x}
      K''(\theta) \, \theta^2 \, \d \theta + M^* 
  \le \int_0^{\theta^\star} \lvert K''(\theta) \rvert \,
      \theta^2 \, \d\theta + M^*
  < \infty \,.
\end{align}
Hence, \eqref{e.main-ineq} and \eqref{e.main-ineq2} continue to hold
and the remainder of the proof proceeds as before.
\end{remark}

\begin{corollary}
Suppose that $K$ satisfies conditions \ref{i.k1}--\ref{i.k3} stated
in the introduction.  Then there exists $x^*<\infty$ so that the
maximal interval of existence of a precipitation ring pattern in the
sense of \eqref{e.abstract-prec-eqn} is $[0,x^*]$.
\end{corollary}

\begin{proof}
When the solution is degenerate, then such $x^*$ exists by definition.
Otherwise, property \ref{i.k1} and the positivity of the kernel on
$(0,1)$ imply that Lemma~\ref{l.infrings} is applicable, i.e., there
exist an infinite number of precipitation rings.  Then, due to
properties \ref{i.k2} and \ref{i.k3}, Theorem~\ref{th.shrink} applies
and asserts the existence of a finite accumulation point $x^*$ of the
ring pattern.
\end{proof}

\section{Existence of degenerate solutions}
\label{s.degenerate}

In this section, we show that degenerate solutions exist.  These are
solutions to \eqref{omega.explicit.0} which cannot be continued past a
finite number of zeros.  While we cannot settle this question for the
concrete kernel introduced in Section~\ref{s.simplified}, we construct
a kernel $K$ such that the solution cannot be continued in the sense
of \eqref{omega.explicit.0} past $x_2$, the end point of the first
precipitation gap.

\begin{theorem}
\label{t.degenerate}
There exist a non-negative kernel $\K \in C([0,1])$ and a constant
$\Gamma$ such that the solution of the integral equation
\eqref{omega.explicit.0} is degenerate.  Moreover, $\K$ is
differentiable at $\theta=0$ and satisfies conditions \ref{i.k1} and
\ref{i.k2}.
\end{theorem}

\begin{remark}
The proof starts from a kernel template which is then modified on a
subinterval $[0,r)$ to produce a kernel $\K$ with the desired
properties.  When the kernel template is $C^1([0,1))$ and
$C^2((0,1))$, the resulting kernel $\K$ will inherit these properties
except possibly at the gluing point $\theta=r$ where continuity of the
first derivative is not enforced.  This is a matter of convenience,
not of principle: The existence of degenerate solutions does not hinge
on the existence of a jump discontinuity for $\K'$.  Straightforward,
yet technical modifications of the gluing construction employed in the
proof will yield a kernel producing degenerate solutions within the
same class of kernels to which Theorem~\ref{th.shrink}, in the sense
of Remark~\ref{r.shrink-generalization}, applies.
\end{remark}

\begin{remark}
In Section~\ref{s.simplified}, we derived a concrete kernel by
simplifying the HHMO-model.  In that setting, there exists an
integrable function $G$ such that $K(\theta) = \theta^2 \, G(\theta)$
and
\begin{gather}
  \Gamma = \int_0^1 \G(\theta) \, \d\theta \,.
  \label{e.gamma-relation}
\end{gather}
In the proof of the theorem, we preserve this relationship, i.e., the
constant $\Gamma$ here will also satisfy \eqref{e.gamma-relation}.
\end{remark}

\begin{proof}
Take any continuous template kernel $\G_\star \colon [0,1] \to \R_+$
with $\G_\star(\theta)\sim k \, \sqrt{1-\theta}$ as $\theta \to 1$ for
some positive constant $k$.  Set
$\K_\star(\theta)=\theta^2 \, \G_\star(\theta)$ and
\begin{equation}
  \Gamma = \int_0^1 \G_\star(\theta) \, \d\theta \,.
\end{equation}
Then $\K_\star(0)=0$ and
\begin{equation}
  \K^\prime_\star(0)
  = \lim_{\theta\searrow0} \frac{\K_\star(\theta)} \theta
  = \lim_{\theta\searrow0} \theta \, \G_\star(\theta)
  = 0 \,.
\end{equation}
Now consider the solution to \eqref{omega.explicit.0} with template
kernel $\K_*$ in place of $K$.  As in the proof of
Lemma~\ref{l.infrings}, the solution must have at least two zeros
$x_1$ and $x_2$.  Let
\begin{equation}
  \omega_\star(x) =
  \begin{dcases} 
    \Gamma - x^2\int_0^1\K_\star(\theta) \, \d\theta & 
    \text{if $x<x_1$} \,, \\
    \Gamma - x^2 \int_0^{x_1/x} \K_\star(\theta) \, \d\theta & 
    \text{otherwise} \,.
  \end{dcases}
\end{equation}
In the following we assume, for simplicity, that
$\omega_\star^\prime(x_2)>0$.  This is true for generic template
kernels $\G_\star$, thus suffices for the construction.  However, it
is also possible to modify the procedure to come to the same
conclusion then $\omega_\star^\prime(x_2)=0$; for details, see
\cite{Darbenas:2018:PhDThesis}.

Note that $\omega_\star$ is continuously differentiable on $[0,x_2]$,
so $\omega^\prime_\star(x_2-\eps)>0$ for all $\eps>0$ small enough.
For each such small $\eps$, set
\begin{equation}
  \omega_{\eps} (x) =
  \begin{dcases}
    \omega_\star(x) & \text{for } x \in [0,x_2-\eps] \,, \\
    \frac12 \, x^2 
      \int_{\tfrac{x_2+\eps}x}^1 \K_\star(\theta) \, \d\theta
    & \text{for } x \in [x_2+\eps,z_\eps] \,,
  \end{dcases}
  \label{e.omega_epsilon}
\end{equation}
where $z_\eps \in (x_2+\eps,x_2+2\eps]$ is chosen such that
$\omega_\eps(z_\eps) \leq \Gamma/2$.  (This is always possible because
$\omega_\eps(x_2+\eps)=0$ and the second case expression in
\eqref{e.omega_epsilon} is continuous, so we can take $z_\eps$ such
that $\omega_\eps(z_\eps)=\Gamma/2$ if such solution exists on
$(x_2+\eps,x_2+2\eps]$, otherwise we take $z_\eps = x_2+2\eps$.)  We
now fill the gap in the definition of $\omega_{\eps}$ such that
\begin{enumerate}[label={\upshape(\alph*)}]
\item $\omega_\eps \colon [0,x_2 + 2 \eps] \to \R$ is continuously
differentiable,
\item \label{i.omega.ii} $\omega_\eps$ is increasing on
$[x_2-\eps,x_2+2\eps]$,
\item \label{i.omega.iii} $\omega_\eps<\Gamma$ on
$[x_2-\eps,x_2+2\eps]$.
\end{enumerate}
Observe that $\omega_\eps(x_2+\eps)=0$.  Moreover, due to
the positivity of $\K_\star$, $\omega_\eps$ is positive and
strictly increasing on the interval $(x_2+\eps,z_\eps]$, and
\begin{equation}
\label{omega.prime}
 \omega^\prime_\eps(x)
 = x\int_{\tfrac{x_2+\eps}x}^1 \K_\star(\theta) \, \d\theta
   + \frac12 \, (x_2+\eps) \,
     \K_\star \Bigl( \frac{x_2+\eps}x \Bigr) \,.
\end{equation} 
Hence, $\omega_\eps^\prime(x_2+\eps)=0$.

We now define a new kernel $\K_\eps \colon [x_1/(x_2+2\eps),1] \to
\R_+$ via
\begin{align}
  \label{K.def.2}
  \K_\eps (\theta) 
  &=\frac{\omega_\eps^\prime(x_1/\theta)}{x_1}
    - 2 \theta \, \frac{\omega_\eps(x_1/\theta) -
    \Gamma}{x_1^2}
\end{align}
and set $\G_\eps(\theta)=\K_\eps(\theta)/\theta^2$.
Due to
\ref{i.omega.ii} and \ref{i.omega.iii} above, $\K_\eps$ is
positive on its interval of definition.  Moreover, for
$x\in(x_1,x_2+2\eps)$,
\begin{equation}
  \K_\eps\Bigl(\frac{x_1}x\Bigr)
  = \frac{x^2}{x_1} \, \frac\d{\d x}
    \frac{\omega_\eps(x)-\Gamma}{x^2} \,.
  \label{e.K.def.1}
\end{equation}
Thus, for $x \in (x_1,x_2-\eps)$ where $\omega_\eps = \omega_\star$,
\begin{equation}
  \K_\eps\Bigl(\frac{x_1}x\Bigr)
  = \frac{x^2}{x_1} \, \frac\d{\d x}
    \frac{\omega_\star(x)-\Gamma}{x^2}
  = \frac{x^2}{x_1} \, \frac\d{\d x}
    \int_{x_1/x}^1\K_\star(\theta) \, \d\theta
  = \K_\star\Bigl(\frac{x_1}x\Bigr) \,,
  \label{e.K.def.3}
\end{equation}
or, equivalently, 
\begin{equation}
  \K_\eps(\theta)=\K_\star(\theta)\text{ for }\theta \in
  \Bigl[\frac{x_1}{x_2-\eps},1\Bigr] \,.
  \label{e.K.def.4}
\end{equation}

Noting that
\begin{subequations}
\begin{gather}
  \I_{[x_1/(x_2+2\eps),1]}(\theta) \, \K_\eps(\theta)
  \to \I_{[x_1/x_2,1]}(\theta) \, \K_\star(\theta)
\intertext{and}
  \I_{[x_1/(x_2+2\eps),1]}(\theta) \, \G_\eps(\theta)
  \to \I_{[x_1/x_2,1]}(\theta) \, \G_\star(\theta)
\end{gather}
\end{subequations}
pointwise for a.e.\ $\theta$ as $\eps\searrow0$, we find that
\begin{multline}
  \frac{\int_0^1 \K_\star(\theta) \, \d\theta
        - \int^1_{x_1/(x_2+2\eps)} \K_\eps(\theta) \, \d\theta}%
       {\int_0^1 \G_\star(\theta) \, \d\theta
        - \int^1_{x_1/(x_2+2\eps)} \G_\eps(\theta) \, \d\theta}
  - \biggl( \frac{x_1}{x_2+2\eps} \biggr)^2
  \to \\
  \frac{\int_0^1 \K_\star(\theta) \, \d\theta
        - \int^1_{x_1/x_2} \K_\star(\theta) \, \d\theta}%
       {\int_0^1 \G_\star(\theta) \, \d\theta
        - \int^1_{x_1/x_2} \G_\star(\theta) \, \d\theta}
  - \frac{x_1^2}{x_2^2}
  = \frac{\int_0^{x_1/x_2} \K_\star(\theta) \, \d\theta}%
         {\int_0^{x_1/x_2} \G_\star(\theta) \, \d\theta}
  - \frac{x_1^2}{x_2^2} \,.
  \label{e.intfrac}
\end{multline}
Recall that $0\le\K_\star(\theta)=\G_\star(\theta) \, \theta^2$, so
that $\K_\star(\theta)\le\G_\star(\theta) \, x_1^2/x_2^2$ on
$[0, x_1/x_2]$.  This implies that the right hand side of
\eqref{e.intfrac} is negative so that there exists $\eps>0$ such
that
\begin{equation}
  \frac{\int_0^1 \K_\star(\theta) \, \d\theta
        - \int^1_{x_1/(x_2+2\eps)}
          \K_{\eps}(\theta) \, \d\theta}%
       {\int_0^1 \G_\star(\theta) \, \d\theta
        - \int^1_{x_1/(x_2+2\eps)}
          \G_{\eps}(\theta) \, \d\theta}
  < \biggl( \frac{x_1}{x_2+2\eps} \biggr)^2 \,.
\end{equation}
In all of the following, we fix $\eps>0$ such that this inequality
holds true, abbreviate $r={x_1}/(x_2+2\eps)$, and set
$\G(\theta) = \G_{\eps}(\theta)$ and $\K(\theta) = \K_{\eps}(\theta)$
for $\theta \in [r,1]$.  We still need to define $\G$ and $\K$ on the
interval $[0,r)$, which is done as follows.

Since $\int_0^r r^{-n} \, \theta^n \, \d \theta\to0$ as $n\to\infty$,
we can choose $n$ such that
\begin{equation}
  r_\star^2
  \equiv
  \frac{\int_0^1 \K_\star(\theta) \, \d\theta
        - \int^1_r \K_{\eps}(\theta) \, \d\theta
        - \G_{\eps}(r) \int_0^r r^{-n} \, \theta^{n+2} \, \d\theta}%
       {\int_0^1 \G_\star(\theta) \, \d\theta
        - \int^1_r \G_{\eps}(\theta) \, \d\theta
        - \G_{\eps}(r) \int_0^r r^{-n} \, \theta^n \, \d \theta}
  < r^2 \,.
  \label{e.rstar2}
\end{equation}
Define $b_1,b_2 \in C^3([0,r],\R_+)$ as the spline functions
\begin{subequations}
  \label{B.spline}
\begin{gather}
  b_1(\theta)
    = \I_{[0,r_\star/2]}(\theta) \,
      \theta^4 \, (r_\star/2-\theta)^4 \,, \\
  b_2(\theta)
    = \I_{[r_\star/2,r_\star]}(\theta) \,
      (r_\star-\theta)^4 \, (\theta-r_\star/2)^4 \,.
\end{gather}
\end{subequations}
By the integral mean value theorem,
\begin{gather}
  \frac{\int_0^r \theta^2 \, b_1(\theta) \, \d \theta}%
       {\int_0^r b_1(\theta) \, \d \theta}
  = \frac{\int_0^{r_\star} \theta^2 \, b_1(\theta) \, \d \theta}%
         {\int_0^{r_\star} b_1(\theta) \, \d \theta}
  < r_\star^2
  < \frac{\int^r_{r_\star} \theta^2 \, b_2(\theta) \, \d \theta}%
         {\int^r_{r_\star} b_2(\theta) \, \d \theta}
  = \frac{\int_0^r \theta^2 \, b_2(\theta) \, \d \theta}%
         {\int_0^r b_2(\theta) \, \d \theta} \,.
\end{gather}
Now define $B_1,B_2 \colon [0,1]\to\R_+$ as
\begin{gather}
  B_1(\lambda)
  = \lambda \int_0^r \theta^2 \, b_1(\theta) \, \d \theta
    + (1-\lambda) \int_0^r \theta^2 \, b_2(\theta) \, \d \theta \,, \\
  B_2(\lambda)
  = \lambda \int_0^r b_1(\theta) \, \d \theta
    + (1-\lambda) \int_0^r b_2(\theta) \, \d \theta \,.
\end{gather}
Clearly, $\frac{B_1(0)}{B_2(0)}>r_\star^2$ and
$\frac{B_1(1)}{B_2(1)}<r_\star^2$. Hence, due to continuity with
respect to $\lambda$, we can find $\lambda_\star \in (0,1)$ such that
\begin{equation}
  \label{lambda.star}
  \frac{B_1(\lambda_\star)}{B_2(\lambda_\star)} = r_\star^2 \,.
\end{equation}
Finally, on the interval $[0,r)$, we define
\begin{gather}
  \K (\theta)
  = \frac{k_\star}{B_1(\lambda_\star)} \,
    \bigl(
      \lambda_\star \, \theta^2 \, b_1(\theta)
      + (1-\lambda_\star) \, \theta^2 \, b_2(\theta)
    \bigr)
    + \G_{\eps}(r) \, \frac{\theta^{n+2}}{r^n} \,,
\end{gather}
where $k_\star$ denotes the numerator of the fraction defining $r_*^2$
in \eqref{e.rstar2}.  Further, we set
$\G(\theta) = \K(\theta)/\theta^2$.  Then $\K$ and $\G$ are continuous
on $[0,1]$, strictly positive on $(0,1)$, and, by direct computation,
satisfy
\begin{subequations}
\begin{gather}
  \int_0^1\K(\theta) \, \d\theta 
  = \int_0^1\K_\star(\theta) \, \d\theta \label{e.k-star-01} \\
\intertext{and}
  \int_0^1\G(\theta) \, \d\theta 
  = \int_0^1\G_\star(\theta) \, \d\theta = \Gamma \,.
\end{gather}
\end{subequations}
Further, using \eqref{e.k-star-01}, \eqref{e.K.def.1},
and the fact that $\omega_\star(x_1)=0$, we verify that 
\begin{equation}
  \omega_{\eps} (x) = 
  \begin{dcases}
    \Gamma - x^2\int_0^1 \K (\theta) \, \d\theta 
    & \text{for } x \in [0,x_1) \,, \\
    \Gamma -x^2 \int_0^{x_1/x} \K(\theta) \, \d\theta 
    & \text{for } x \in [x_1, x_2+2\eps] \,.
  \end{dcases}
  \label{e.omega-newdef}
\end{equation}
on the entire interval $[0, x_2+2\eps]$.  Moreover, comparing
\eqref{e.omega-newdef} with \eqref{omega.explicit.0} and noting that
$x_1$ is the first and $x_2+\eps$ the second zero of $\omega_\eps$ by
construction, we see that $\omega_{\eps}$ satisfies
\eqref{omega.explicit.0} at least on the interval $[0,x_2+\eps]$.

To complete the proof, we show that it is not possible to find a
solution $\omega$ to \eqref{omega.explicit.0} that extends to a
non-degenerate solution past the interval $[0,x_2+\eps]$ on which
$\omega = \omega_{\eps}$.  Assume that, on the contrary, such an
extension exists.  Then there exists a small interval
$I=(x_2+\eps,x_2+\eps+\eps_\star)$ on which $\omega$ is either
positive or negative.  (Note that we must require that
$\eps_\star<z_{\eps}$, cf.\ \eqref{e.omega_epsilon} where $z_\eps$ is
first introduced.)  Suppose first that $\omega<0$ on $I$.  Then
\eqref{e.omega-newdef} continues to provide a solution for
\eqref{omega.explicit.0} on $I$, but $\omega_{\eps}$ is positive
there, a contradiction.  Suppose then that $\omega>0$ on $I$.  Then,
using \eqref{e.abstract-prec-eqn} and \eqref{e.omega_epsilon}, we
express $\omega$ on the interval $I$ as
\begin{align}
  \omega(x)
  & = \Gamma - x^2 \int_0^{\tfrac{x_1}x} \K(\theta) \, \d\theta
      - x^2\int_{\tfrac{x_2+\eps}x}^1 \K (\theta) \, \d\theta
      \notag\\
  & = \omega_{\eps}(x)
      - x^2 \int_{\tfrac{x_2+\eps}x}^1 \K (\theta) \, \d\theta
      \notag\\
  & = - \frac12 \, x^2 \int_{\tfrac{x_2+\eps}x}^1
      \K_\star(\theta) \, \d\theta < 0 \,.
 \label{e.extension-positive}
\end{align}
The last equality is due to \eqref{e.omega_epsilon} and
\eqref{e.K.def.4} which is applicable for sufficiently small
$\eps_\star$.  Again, this contradicts the assumed sign of $\omega$.
Thus, we conclude that $\omega$ cannot be extended via formula
\eqref{e.abstract-prec-eqn} onto any right neighborhood of
$x=x_2+\eps$.
\end{proof}

\section{Extended solutions for the simplified HHMO-model}
\label{exist.simplified}

So far, we have seen that precipitation band patterns as a sequence of
intervals in which $\omega(x)>0$, i.e.\ the reactant concentration
exceeds the super-saturation threshold, must break down at a finite
location $x^*$ which is either an accumulation point given by
Theorem~\ref{th.shrink}, or until a point at which the solution
degenerates after a finite number of precipitation bands as in
Theorem~\ref{t.degenerate}.  In this section, we consider a more
general notion of solution which is motivated by the construction of
weak solutions to the full HHMO-model in
\cite{HilhorstHM:2009:MathematicalSO}.

\begin{definition}
\label{d.extended}
A pair $(\omega, \rho)$ is an \emph{extended solution} of the
simplified HHMO-model if $\omega \in C([0,\infty))$,
\begin{equation}
\label{HHMO.extended}
  \omega(x)
  = \Gamma-x^2\int_0^1K(\theta) \, \rho(x\theta) \, \d\theta \,,
\end{equation}
and $\rho$ is a measurable function on $[0,\infty)$ taking values from
the Heaviside graph, i.e.,
\begin{equation}
\label{Heaviside}
  \rho(y)\in H(\omega(y))
  = \begin{cases}
      0 & \text{if }\omega(y)<0 \,, \\
      [0,1] & \text{if }\omega(y)=0 \,, \\
      1 & \text{otherwise} \,.
    \end{cases}
\end{equation}
\end{definition}

\begin{theorem}
If $K \in C([0,1])$, then an extended solution to
\eqref{HHMO.extended} exists.
\end{theorem}

\begin{proof}
Changing variables, we write \eqref{HHMO.extended} as
\begin{equation}
  \omega(x) = \Gamma
  - x \int_0^x K \Bigl( \frac y x \Bigr) \, \rho(y) \, \d y \,.
\end{equation}
Now consider a family of mollified Heaviside functions
$H_\eps\in C^\infty(\R,[0,1])$ parameterized by $\eps>0$ such that
$H_\eps(z)=1$ for $z\ge\eps$ and $H_\eps(z)=0$ for $z\le-\eps$.  We
claim that, for fixed $\eps>0$, the corresponding mollified equation
\begin{equation}
  \label{omega.epsilon}
  \omega_\eps(x)
  = \Gamma - x \int_0^x
    K \Bigl( \frac y x \Bigr) \, H_\eps(\omega_\eps(y)) \, \d y
\end{equation}
has a solution $\omega_\eps \in C([0,\infty))$.  Indeed, suppose that
$\omega_\eps$ is already defined on some interval $[0,a]$, where $a$
may be zero.  We seek $\omega_\eps \in C([a,a+\delta])$ which
continuously extends $\omega_\eps$ past $x=a$ as a fixed point of a
map $T$ from $C([a,a+\delta])$ endowed with the supremum norm into
itself, defined by
\begin{equation}
  T[\phi] (x)
  = \Gamma-x\int_0^a K\left(\frac{y}x \right) \,
      H_\eps(\omega_\eps(y)) \, \d y
    - x\int_a^x K\left(\frac{y}x \right) \,
      H_\eps(\phi(y)) \, \d y \,.
\end{equation}
Since
\begin{align}
  \bigl| T[\phi](x) - T[\psi](x) \bigr|
  & = \biggl|
        x \int_a^x K \Bigl( \frac y x \Bigr) \,
        \bigl( H_\eps(\psi(y)) - H_\eps(\phi(y)) \bigr) \, \d y
      \biggr|
      \notag \\
  & \leq x \int_a^x \Bigl| K \Bigl( \frac y x \Bigr) \Bigr| \,
      \norm{}{H_\eps(\phi) - H_\eps(\psi)}{L^\infty} \, \d y
      \notag \\
  & \leq x \, (x-a) \, \norm{}{K}{L^\infty} \,
      \norm{}{H_\eps^\prime}{L^\infty} \,
      \norm{}{\phi - \psi}{L^\infty} \,,
\end{align}
$T$ is a strict contraction for $\delta>0$ small enough, hence has a
unique fixed point.  In addition, the maximal interval of existence of
$\omega_\eps$ is closed, as the right hand side of
\eqref{omega.epsilon} is continuous, and open at the same time due to
the preceding argument.  Thus, a solution
$\omega_\eps \in C([0,\infty))$ exists (and is unique).

By direct inspection, for every fixed $b>0$, the families
$\{\omega_\eps\}$ and $\{H_\eps\circ\omega_\eps\}$ are uniformly
bounded in $C([0,b])$ endowed with the supremum norm.  Moreover,
$\{\omega_\eps\}$ is equicontinuous.  Indeed, for $y,z \in (0,b]$,
\begin{equation}
  \biggl|
    \frac{\omega_{\eps_i}(z)-\Gamma}{z}
    - \frac{\omega_{\eps_i}(y)-\Gamma}y
  \biggr|
  \leq \int_0^{\max \{y,z\}}
    \biggl|
      K \Bigl(\frac \theta{z} \Bigr) - K \Bigl(\frac\theta{y} \Bigr)
    \biggr| \, \d \theta \,,
\end{equation}
where, by the dominated convergence theorem, the right hand side
converges to zero as $z \to y$.  Equicontinuity at $y=0$ is obvious.
Thus, by the Arzel\'a--Ascoli theorem, there exist a decreasing
sequence $\eps_i \to 0$ and a function $\omega \in C([0,b])$ such that
$\omega_{\eps_i} \to \omega$ in $C([0,b])$.  Further, by the
Banach--Alaoglu theorem, there exists $\rho \in L^\infty([0,b])$ such
that, possibly passing to a subsequence,
$H_{\eps_i}\circ\omega_{\eps_i} \rightharpoonup \rho$ weakly-$*$ in
$L^\infty$.  This implies that $\rho$ takes values a.e.\ from the
interval $[0,1]$ which contains the convex hull of the sequence.
Passing to the limit in \eqref{omega.epsilon}, we conclude that
\begin{equation}
  \omega(x) = \Gamma
  - x \int_0^x K \Bigl( \frac y x \Bigr) \, \rho(y) \, \d y  \,.
\end{equation}
Finally, we claim that $\rho(y)=1$ whenever $\omega(y)>0$.  Indeed,
fixing $y$ such that $\omega(y)>0$, equicontinuity of $\omega_\eps$
implies that there exists a neighborhood of $y$ on which
$\omega_{\eps_i}$ is eventually strictly positive.  On this
neighborhood, $H_{\eps_i}\circ\omega_{\eps_i}$ converges strongly to
$1$.  A similar argument proves that $\rho(y)=0$ whenever
$\omega(y)<0$.

So far, we have shown that $(\omega,\rho)$ satisfy
\eqref{HHMO.extended} and \eqref{Heaviside} on $[0,b]$.  To extend the
interval of existence, we can iteratively restart the compactness
argument on intervals $[0,nb]$ for $n\in \N$, passing to a subsequence
each time.  The proves existence of an extended solution on
$[0,\infty)$.
\end{proof}

Uniqueness of extended solutions is a much more delicate issue.  In
the following particular case, we can give a positive answer to the
question of uniqueness.

\begin{definition}
\label{reg.extend}
An extended solution $(\omega,\rho)$ to the simplified HHMO-model is
\emph{regularly extended} to an interval $[x^\star,x^\star+\eps]$,
where $x^*$ is the point of breakdown in the sense of
Theorem~\ref{th.shrink} or Theorem~\ref{t.degenerate} and $\eps>0$, if
$\omega \equiv 0$ on this interval.
\end{definition}

If $(\omega_1, \rho_1)$ and $(\omega_2,\rho_2)$ are pairs of regularly
extended solutions to \eqref{HHMO.extended}, then $\omega_1$ and
$\omega_2$ coincide on $[0,x^*]$ by construction and on
$[x^\star,x^\star+\eps]$ by definition.  Moreover,
$\Delta \rho=\rho_1-\rho_2 = 0$ on $[0,x^\star]$.  Thus, the question
of uniqueness reduces to a statement on the non-existence of
non-trivial solutions to the linear homogeneous integral equation
\begin{equation}
  \label{WDCVIE}
  \int_0^1K(\theta) \, \Delta \rho(x\theta) \, \d \theta = 0 
\end{equation}
for $x\in[0,x^\star+\eps]$.  Due to properties \ref{i.k1}--\ref{i.k3}
of the kernel, the problem falls into the general class of weakly
degenerate cordial Volterra integral equations.  In
\cite{DarbenasO:2019:UniquenessSW}, we answer this question in the
affirmative.  

While we believe that extended solutions are generically regularly
extended, we cannot exclude the possibility that extended solutions
develop a precipitation band pattern that accumulates at $x^*$ from
above.  Thus, the general question of unique extendability remains
open.

We conjecture that the question of uniqueness of extended solutions
might be addressed by replacing Definition~\ref{d.extended} by a
formulation in terms of a mixed linear complementarity problem.  To be
concrete, write $\omega = \omega_+ - \omega_-$, where $\omega_+$ is
the positive part and $\omega_-$ the negative part of $\omega$.
Further, set $\sigma = 1-\rho$ and define the vector functions
\begin{equation}
  V =
  \begin{pmatrix}
    \sigma \\ \rho
  \end{pmatrix}
  \qquad \text{and} \qquad
  W =
  \begin{pmatrix}
    \omega_+ \\ \omega_-
  \end{pmatrix} \,.
\end{equation}
Then we can formulate the extended solution as follows.  Find $V \geq
0$, $W \geq 0$ such that
\begin{subequations}
\begin{equation}
  \mathcal L V + \mathcal M W + B = 0
\end{equation}
subject to 
\begin{equation}
  \langle V, W \rangle = 0 \,,
  \label{e.complementarity-condition}
\end{equation}
where $\mathcal L$ and $\mathcal M$ are linear operators defined by
\begin{gather}
  \mathcal L V =
  \begin{pmatrix}
    \displaystyle x^2 \int_0^1 K(\theta) \, \rho(x\theta) \, \d \theta \\
    \rho + \sigma
  \end{pmatrix} \,, \\
  \mathcal M W = 
  \begin{pmatrix}
    \omega_+ - \omega_- \\
    0
  \end{pmatrix} \,,
\end{gather}
and
\begin{gather}
  B =
  \begin{pmatrix}
    \Gamma \\ -1
  \end{pmatrix} \,.
\end{gather}
\end{subequations}
The angle brackets in \eqref{e.complementarity-condition} denote the
canonical inner product for vectors of $L^2((0,a))$-functions,
\begin{equation}
  \langle V, W \rangle
  = \int_0^a \omega_+ (x) \, \sigma(x) \, \d x
    + \int_0^a \omega_- (x) \, \rho(x) \, \d x \,.
\end{equation}
This formulation is known as a mixed linear complementarity problem.
To make progress here, it is necessary to adapt Lions--Stampaccia
theory \cite{LionsS:1967:VariationalI} to the case of mixed
complementarity problems; see, e.g.,
\cite{CryerD:1980:EquivalenceLC,ZengAY:2009:EquivalenceML}, for
different reformulations in a Sobolev space setting.

\section{Discussion}
\label{s.discussion}

In this paper, we have identified a mechanism which leads to very
rapid, i.e., finite-time equilibrization of a dynamical system with
memory that is ``damped'' via relay hysteresis.  Past a certain point,
a solution can only be continued in a generalized sense by
``completing the relay''.  We can assert existence and conditional
uniqueness for the generalized solution, but full well-posedness
remains open.  Possible approaches are a reformulation of the concept
of generalized solution in terms of a mixed linear complementarity
problem as outlined above, or possibly a fixed point formulation using
fractional integral operators.  We believe that the integral equation
\eqref{omega.explicit.0} is a useful test bed for studying such
approaches, with the hope to eventually transfer results to more
general reaction-diffusion equations with relay hysteresis.

The detailed observed behavior is very much tied to property
\ref{i.k1}, the square-root degeneracy of the kernel $K$ near
$\theta=1$.  From the perspective of solving integral equations, this
behavior is too degenerate for classical contraction mapping arguments
as used by Volterra to apply, but it is sufficiently non-degenerate
that strong results can still be proved, see the discussion in
\cite{DarbenasO:2019:UniquenessSW}.  This degeneracy is associated
with the scaling behavior of the heat kernel, so even when an exact
reduction from a PDE to an integral equation, as is possible for the
simplified HHMO-model, is not available, the associated phenomenology
is expected to survive.

Let us finally remark on the connection of our models to the
real-world Liesegang precipitation phenomenon.  Our detailed results
on the breakdown of patterns for the simplified model imply that we
cannot expect the Keller--Rubinow family of models to provide a good
literal description of Liesegang rings.  However, the fact that,
independent of the details of simplified vs.\ full dynamics, the
models converge rapidly toward a steady-state which \emph{only exists
as a generalized solution}, we believe that it might be possible to
interpret fractional values of the precipitation function as a
\emph{precipitation density}.  In this view, the model would provide a
coarse-grained description of the phenomenon in the sense that the
precise information about the location of the rings is lost, but the
local average fraction of space covered by precipitation rings can
still be asserted.  If this suggested interpretation is valid, the
explicit asymptotic profile detailed in Section~\ref{s.simplified}
will provide a direct relationship between the parameters of the
system and the \emph{macroscopic} properties of the Liesegang pattern.

This point of view is different from that taken by Duley \emph{et
al.}\ \cite{DuleyFM:2019:RegularizationOS}.  As they numerically
encountered a qualitatively similar breakdown of solutions in the full
Keller--Rubinow model (the model without the fast reaction limit
taken), they chose to modify the dynamics of the model by introducing
a delay variable to smooth out the onset of precipitation.  Their
approach aims at a more physical description at the \emph{microscopic}
level, i.e., toward a model that can correctly represent the width and
location of each individual Liesegang ring.

\appendix
\section{Properties of the simplified HHMO-kernel}
\label{kernel.prop}

In the following, we prove a collection of results on the asymptotic
behavior of the kernel $G$ as given in \eqref{e.G} which imply
properties \ref{i.k1} and \ref{i.k2}.  As a corollary, we obtain that
$G$ is integrable, i.e., that the integrals in \eqref{omega.explicit}
and \eqref{e.Gamma} are finite.  We have verified property \ref{i.k3}
numerically, which is easily done to machine accuracy.  A proof seems
feasible, but would be rather involved and tedious, and does not offer
further insight into the problem.

We begin by observing that $G$ is clearly continuous on $(-1,0)$ and
$(0,1)$.  Thus, we focus on the local asymptotics of $G$ near
$\theta=1$, $\theta=-1$ and $\theta=0$.

\begin{lemma} \label{l.g1}
$\displaystyle G(\theta) \sim \sqrt{\frac{2}{\pi}} \,
\frac{u^*}{\alpha} \, \sqrt{1-\theta}$ as $\theta\to1$.
\end{lemma}

\begin{proof}
Rearranging expression \eqref{e.G}, applying the integral mean value
theorem, and setting $\sigma^2 = \frac{4(\zeta^2 - \alpha^2 \,
\theta^2)}{\alpha^4 \, \theta^2(1-\theta)^2}$, we obtain
\begin{align}
  G(\theta)
  & = \frac1{\sqrt \pi} \int_{\alpha\lvert\theta\rvert}^\alpha
      \frac{\alpha\zeta}{\sqrt{\zeta^2-\alpha^2 \, \theta^2}} \,
	\exp \biggl(
               - \frac{\alpha^4 \, \theta^2 \, (1-\theta)^2}%
                      {4 \, (\zeta^2 - \alpha^2 \, \theta^2)}
             \biggr)
        \exp \biggl(
               - \frac{\alpha^2 \, (1-\theta)^2}%
                      {4}
             \biggr) \,
        \frac{\Phi (\zeta)}{\zeta^3} \, \d \zeta 
    \notag \\
  & = \frac{\Phi(\zeta_\theta)}{\sqrt \pi \, \zeta_\theta^3} \,
      \exp \biggl(
               - \frac{\alpha^2 \, (1-\theta)^2}{4}
	   \biggr)
      \int_{\alpha\lvert\theta\rvert}^\alpha
        \frac{\alpha\zeta}{\sqrt{\zeta^2-\alpha^2 \, \theta^2}} \,
	\exp \biggl(
               - \frac{\alpha^4 \, \theta^2 \, (1-\theta)^2}%
                      {4 \, (\zeta^2 - \alpha^2 \, \theta^2)}
             \biggr) \, \d \zeta
    \notag \\
  & = \frac{\Phi(\zeta_\theta) \, \alpha^3}%
           {2 \sqrt \pi \, \zeta_\theta^3} \,
      \exp \biggl(
               - \frac{\alpha^2 \, (1-\theta)^2}{4}
	   \biggr) \, \lvert \theta \rvert \, (1-\theta) \, 
      \int_0^{z(\theta)}
	\exp \biggl(
               - \frac1{\sigma^2}
             \biggr) \, \d \sigma 
  \label{e.G.alt}
\end{align}
for some $\zeta_\theta \in [\alpha \lvert\theta\rvert, \alpha]$ and
\begin{equation}
  z(\theta) = \frac2{\alpha \, |\theta|} \, 
              \sqrt{\frac{1+\theta}{1-\theta}} \,.
  \label{e.ztheta}
\end{equation} 
When $\theta \to 1$, we have $z(\theta) \to \infty$.  Then, as averages
converge to asymptotic values,
\begin{equation}
  \label{sim.av}
  \frac1{z(\theta)} \int_0^{z(\theta)}
	\exp \biggl(
               - \frac1%
                      {\sigma^2}
             \biggr) \, \d \sigma \to 1 \,.
\end{equation}
For the prefactor in \eqref{e.G.alt}, we observe that
\begin{equation}
  \lim_{\theta \to 1}
    \frac{\Phi(\zeta_\theta) \, \alpha^3}%
           {2 \sqrt \pi \, \zeta_\theta^3} \,
      \exp \biggl(
               - \frac{\alpha^2 \, (1-\theta)^2}{4}
	   \biggr) \, \lvert \theta \rvert \, \sqrt{1-\theta} \,
           z(\theta) 
    = \sqrt{\frac2\pi} \, \frac{\Phi(\alpha)}\alpha \,.
  \label{e.prefactor}
\end{equation}
Since $\Phi(\alpha)=u^*$, this altogether implies the claim.
\end{proof}

\begin{lemma}
$\displaystyle G(\theta) \sim \sqrt{\frac{2}{\pi}} \,
\frac{u^*}{\alpha^3} \, \exp \biggl( \frac{\alpha^2}4 \biggr) \,
(1+\theta)^{\tfrac32} \, \exp \biggl( -\frac{\alpha^2}{2(1+\theta)}
\biggr)$ as $\theta \to -1$.
\end{lemma}

\begin{proof}
First note that 
\begin{equation}
  \int_0^{z}
    \exp \biggl(
           - \frac1{\sigma^2}
         \biggr) \, \d \sigma 
  \sim \frac{z^3}{2} \, \exp \biggl( -\frac1{z^2} \biggr)
\end{equation}
as $z \to 0$.  Thus, in the limit $\theta \to -1$ with $z(\theta)$
given by \eqref{e.ztheta},
\begin{align}
  \int_0^{z(\theta)}
    \exp \biggl(
           - \frac1{\sigma^2}
         \biggr) \, \d \sigma 
  & \sim \frac{4}{\alpha^3 \, \lvert \theta \rvert^3} \, 
       \biggl(
         \frac{1+\theta}{1-\theta}
       \biggr)^{\tfrac32} \, 
       \exp \biggl( 
              - \frac{\alpha^2 \, \theta^2}4 \, 
                \frac{1-\theta}{1+\theta}
            \biggr) 
    \notag \\
  & \sim \frac{\sqrt 2}{\alpha^3} \, (1+\theta)^{\tfrac32} \, 
       \exp \biggl( 
              \frac54 \, \alpha^2
            \biggr) \,
       \exp \biggl( 
              - \frac{\alpha^2}{2(1+\theta)}
            \biggr) \,.
  \label{e.intasymptotics}
\end{align}
Then, using expression \eqref{e.G.alt} for $G(\theta)$ as in the proof
of Lemma~\ref{l.g1}, noting that the limit of the prefactor can be
obtained by direct substitution, and using \eqref{e.intasymptotics},
we obtain the claim.
\end{proof}

Hence, $G$ is continuous on $[-1,0)$ and $(0,1]$.  We next determine
the asymptotic behavior of $G$ at $\theta=0$ depending on $\kappa$.
To simplify notation, we write
\begin{equation}
  C_1 = \frac{u^*}{\alpha^\kappa
  M \bigl(\frac\kappa2,\kappa+\frac12,-\frac{\alpha^2}4 \bigr)} \,.
  \label{e.c1}
\end{equation}
to denote the constant prefactor in the explicit expression
\eqref{e.phi.gamma} for $\Phi$ in the case $\eta<\alpha$.

\begin{lemma} \label{l.asymptotic}
\label{kappa}
When $\kappa>2$, $G$ extends to a continuous function on $[-1,1]$ with
\begin{equation}
  G(0) = \frac\alpha{\sqrt \pi} \, 
         \exp \biggl( -\frac{\alpha^2}4 \biggr) 
         \int_0^\alpha \frac{\Phi(\zeta)}{\zeta^3} \, \d \zeta \,.
\end{equation}
When $1<\kappa<2$, as $\theta \to 0$,
\begin{equation}
  G(\theta) 
  \sim \lvert\theta\rvert^{\kappa-2} \, C_1 \,
       \frac{\alpha^{\kappa-1}}{\sqrt\pi}
       \int_0^1
         \exp \biggl( 
           - \frac{\alpha^2}%
                  {4 \, \sigma^2}
         \biggr) \,
         (1-\sigma^2)^{-\tfrac\kappa2} \, \d \sigma \,.
  \label{e.kappa1-2}
\end{equation} 
Finally, when
$\kappa=2$, as $\theta \to 0$,
\begin{equation}
  G(\theta) 
  \sim - C_1 \, \frac\alpha{\sqrt\pi}
         \exp \biggl( 
           - \frac{\alpha^2}{4}
         \biggr) \,
         \ln \lvert \theta \rvert \,.
  \label{e.kappa=2}
\end{equation}
\end{lemma}

\begin{proof}
We re-write \eqref{e.G} as
\begin{equation}
  G(\theta)
  = \frac\alpha{\sqrt \pi} \int_0^\alpha f(\theta,\zeta) \, \d \zeta
  \label{e.g-alt2}
\end{equation}
with 
\begin{equation}
  f(\theta, \zeta)
  =  \I_{[\alpha\lvert\theta\rvert,\alpha]}(\zeta) \,
      \frac1{\sqrt{1 - \frac{\alpha^2\theta^2}{\zeta^2}}} \,
      \exp \biggl(
             - \frac{\alpha^2 \, (1-\theta)^2}%
                    {4 \, \bigl(1 - \frac{\alpha^2\theta^2}{\zeta^2} \bigr)}
           \biggr) \,
      \frac{\Phi(\zeta)}{\zeta^3} \,.
\end{equation}
For $\theta<\tfrac12$,
\begin{equation}
  f(\theta, \zeta)
  \leq \sup_{z \in (0, \alpha^{-1})}
      \frac1z \, \exp \biggl( - \frac{\alpha^2}{16 \, z^2} \biggr) \,
      \frac{\Phi(\zeta)}{\zeta^3} \,.
\end{equation}
As the Kummer function in the expression for
$\Phi(\zeta)$ from \eqref{e.phi.gamma} limits to $1$ as $\zeta \to
0$, this upper bound is integrable for
$\kappa>2$.  The first case claimed in the lemma is thus a direct
consequence of the dominated convergence theorem.

When $1<\kappa\leq 2$, the integral in \eqref{e.g-alt2} is divergent
as $\theta \to 0$.  To determine its asymptotics, we apply the change
of variables
\begin{equation}
  \sigma^2 = 1-\frac{\alpha^2 \, \theta^2}{\zeta^2}
\end{equation}
and insert the explicit expression for $\Phi$, so that
\begin{equation}
  G(\theta) 
  = C_1 \, \frac\alpha{\sqrt\pi} \, (\alpha \theta)^{\kappa-2} 
    \int_0^{\sqrt{1-\theta^2}} g(\sigma; \sqrt{1-\theta^2}) \, 
      (1-\sigma)^{-\tfrac\kappa2} \, \d \sigma 
  \label{e.G-alt3}
\end{equation}
with
\begin{equation}
  g(\sigma; \xi)
  = \exp \biggl( 
           - \frac{\alpha^2 \, \bigl(1-\sqrt{1-\xi^2} \bigr)^2}%
                  {4 \, \sigma^2}
         \biggr) \,
    M \biggl(
        \frac\kappa2, \kappa + \frac12, 
        - \frac{\alpha^2 \, (1-\xi^2)}{4 \, (1-\sigma^2)}
      \biggr) \, 
    (1+\sigma)^{-\tfrac\kappa2} \,.
\end{equation}
On the closed unit square, $g$ is strictly positive and bounded.

When $1<\kappa<2$, the integrand in \eqref{e.G-alt3} is integrable on
$[0,1]$ so that once again the dominated convergence theorem applies,
implying \eqref{e.kappa1-2}.
  
When $\kappa=2$, we split the expression for $g$ into two terms,
\begin{align}\nonumber
  g(\sigma; \xi)
  & = \exp \biggl( 
           - \frac{\alpha^2 \, \bigl(1-\sqrt{1-\xi^2} \bigr)^2}%
                  {4 \, \sigma^2}
           \biggr) \, \frac1{1+\sigma}
      \notag \\
  & \quad - \frac{\alpha^2}{10} \,
      \exp \biggl( 
            - \frac{\alpha^2 \, \bigl(1-\sqrt{1-\xi^2} \bigr)^2}%
                 {4 \, \sigma^2}
           \biggr) \, \frac{1-\xi^2}{1-\sigma^2} \,
         M \biggl(
             1, \frac72, 
             -\frac{\alpha^2 \, (1-\xi^2)}{4 \, (1-\sigma^2)}
           \biggr) \, \frac1{1+\sigma}
      \notag \\
   & \equiv T_1 (\sigma; \xi)
            + \frac{1-\xi}{1-\sigma} \, T_2 (\sigma; \xi) \,.  
\end{align} 
Clearly,
\begin{equation}
  \int_0^\xi \frac{T_1(\sigma; \xi)}{1-\sigma} \, \d \sigma
  \sim - \frac12 \, 
         \exp \biggl( -\frac{\alpha^2}4 \biggr) \, \ln(1-\xi) 
  \sim - \exp \biggl( -\frac{\alpha^2}4 \biggr) \, 
         \ln \abs{\theta} \,.
  \label{e.intsim}
\end{equation}
Moreover, using mean value theorem, for some $\sigma_\xi\in(0,\xi)$ we obtain
\begin{align}
  \int_0^\xi \frac{1-\xi}{(1-\sigma)^2} \, T_2(\sigma;\xi) \, \d\sigma
  = (1-\xi) \, T_2(\sigma_\xi; \xi)
    \int_0^\xi \frac{\d\sigma}{(1-\sigma)^2}
  = \xi \, T_2(\sigma_\xi; \xi) \,.
\end{align}
As $T_2$ is bounded, the contribution from $T_1$ is asymptotically
dominant as $\theta \to 0$.  This implies \eqref{e.kappa=2}.
\end{proof}

\begin{corollary}
The integrals in \eqref{omega.explicit} and \eqref{e.Gamma} are finite.
\end{corollary}
\begin{proof}
We know that $G$ is continuous on $[-1,0)$ and $(0,1]$.  The
asymptotic behavior of $G$ near $\theta=0$ is given by
Lemma~\ref{l.asymptotic}.  We note that the functions
$\ln\abs{\theta}$ and $\abs{\theta}^{\kappa-2}$ are integrable on any
neighborhood of $\theta=0$, therefore $G$ and $K$ are integrable on
$[-1,1]$.
\end{proof}

\begin{corollary}
\label{K.properties}
The kernel $K$ associated with equation \eqref{omega.explicit}
satisfies properties \textup{(i)} and \textup{(ii)}.
\end{corollary}
\begin{proof}
We can use Lemma~\ref{kappa} for all $\kappa>1$ to obtain 
\begin{equation}
  K(0)=\lim_{\theta\searrow0}K(\theta)
  =\lim_{\theta\searrow0}\theta^2(G(\theta)+G(-\theta))=0 \,.
\end{equation} 
Furthermore, we note that the same lemma implies
\begin{equation}
  K'(0)=\lim_{\theta\searrow0}\frac{K(\theta)-K(0)}{\theta-0}
  = \lim_{\theta\searrow0}\theta (G(\theta)+G(-\theta))=0 \,.
\end{equation}
This proves \textup{(i)}.  Property \textup{(ii)} is a direct
consequence of Lemma~\ref{l.g1}.
\end{proof}

\section{Numerical schemes}
\label{a.numerics}

We solve the full HHMO-model \eqref{e.u-phi} and the simplified
HHMO-model \eqref{e.simplified} in similarity variables, where the
full model reads
\begin{subequations}
  \label{e.full-similarity}
\begin{gather}
  s \, w_s - \eta \, w_\eta - 2 \, w_{\eta \eta}
  = \frac{2 \gamma}{\eta^2} \, H(\alpha - \eta) \, \Phi
    - 2 \, s^2 \, p \, (\Phi + w)  \,, \label{e.full-similarity.a} \\
  w_\eta(0,s) = 0 \quad \text{for } s \geq 0 \,, \\
  w(\eta,s) \to 0 \quad \text{as } \eta \to \infty
  \text{ for } s \geq 0 \,.
\end{gather}
Toward deriving a scheme, we \emph{impose} a binary precipitation
function so that $p(\eta,s)=1$ if and only if
$w(\eta',s') \geq u^* - \Phi(\eta')$ anywhere along the characteristic
line $\eta s = \eta' s'$ with $s' \leq s$.  Note that the change of
variables maps the origin of the $x$-$t$ plane onto the entire
$\eta$-axis.  Since $\sup_x \lvert u(x,t) - \psi(x,t)\rvert$ as
$t \searrow 0$, we initialize with $u(\eta,0) = \Psi(\eta)$, so that
\begin{equation}
  w(\eta,0) = \Psi(\eta) - \Phi(\eta)
  \quad \text{for } \eta \geq 0 \,.
\end{equation}
\end{subequations}
The corresponding evolution equation for the simplified model is 
\begin{equation}
  s \, w_s - \eta \, w_\eta - 2 \, w_{\eta \eta}
  = \frac{2 \gamma}{\eta^2} \, H(\alpha - \eta) \, \Phi
    - 2 \, s^2 \, p \, \Phi
  \label{e.simplified-similarity.a}
\end{equation}
where $p$ is binary as before with $p(\eta,s)=1$ if and only if
$w(\eta',s') \geq u^* - \Phi(\alpha)$ for $\eta s = \alpha s'$ with
$\eta \leq \alpha$.  Boundary and initial condition are as for the
full model.

The computational domain is taken as the finite interval
$[0, 6\alpha]$, where the factor $6$ is ensures sufficient accuracy
for the essential part of the solution.  Let $N$ denote the grid cell
at $\eta = \alpha$, $N_\full = 6 N$ the number of spatial mesh points,
and $\Delta \eta = \alpha/N$ the spatial mesh size.  Then the spatial
mesh points are given by $\eta_i = i \Delta \eta$ for
$i = 0, \dots, N_\full-1$.  Similarly, the temporal mesh points are
given by $s_j = j \Delta s$.  We write $w_i^j$ to denote the numerical
approximation to $w(\eta_i,s_j)$, $p_i^j$ to denote the numerical
approximation to $p(\eta_i,s_j)$, and set $\Phi_i = \Phi(\eta_i)$.  We
use a first order upwind finite difference for the advection term and
the standard second order finite difference approximation for the
Laplacian.  The boundary conditions are represented by
$w_{-1}^j=w_0^j$ and $w_{6N+1}^j=0$, respectively.  In time, we use a
simply first-order difference where the left hand side of the
evolution equation is treated implicitly and the right hand side
explicitly.

This leads to solving, at each time step, a system of linear equations
$A^j w^{j} = b^{j-1}$ where $A^j$ is the tridiagonal matrix with
coefficients
\begin{subequations}
\begin{gather}
  a_{i,i}^j = j + i + 4/\Delta \eta^2
  \qquad \text{for } i = 0, \dots, N_\full -1 \,, \\
  a_{i,i-1}^j = - 2/\Delta \eta^2
  \qquad \text{for } i = 1, \dots, N_\full - 1 \,, \\
  a_{i,i+1}^j = - i - 2/\Delta \eta^2
  \qquad \text{for } i = 1, \dots, N_\full - 2 \,, \\
  a_{0,1}^j = -4/\Delta \eta^2  \,,
\end{gather}
and $b^{j-1}$ is the vector whose coefficients are given for the full
model by
\begin{gather}
  b_i^{j-1}
  = \frac{2 \gamma}{\eta_i^2} \, H_{i \leq N} \, \Phi_i
    + j \, w_i^{j-1}
    - 2 \, s_j^2 \, p_i^{j-1} \, (\Phi_i + w_i^{j-1})
    \label{e.bij}
\end{gather}
\end{subequations}
for $i = 1, \dots, N_\full -1$.  For the simplified model,
expression \eqref{e.bij} is modified to read
\begin{gather}
  b_i^{j-1}
  = \frac{2 \gamma}{\eta_i^2} \, H_{i \leq N} \, \Phi_i
    + j \, w_i^{j-1}
    - 2 \, s_j^2 \, p_i^{j-1} \, \Phi_i \,.
\end{gather}
As the first term on the right of \eqref{e.full-similarity.a} resp.\
\eqref{e.simplified-similarity.a} diverges for $\eta \searrow 0$ for
$\gamma \in (0,2)$, which true for the parameters used in this
paper, we set $b_0^{j-1} = b_1^{j-1}$ in both cases; the resulting
solution, however, is insensitive to the actual value used.  This
treatment is justfied by observing that the resulting scheme is
sufficiently accurate when applied to the analytically known
self-similar solution $\Phi$.

We use the following simple transport scheme for the precipitation
function $p$.  For both full and simplified model, and spatial indices
$i<N$ corresponding to $\eta < \alpha$, we only need to transport the
values of the precipitation function along the characteristic curves.
We note that the characteristic curves define a map from the temporal
interval $[0,s]$ at $\eta=\alpha$ to the spatial interval $[0,\alpha]$
at time $s=j \Delta s$.  This map scales each grid cell by a factor
$N/j$.  We distinguish two sub-cases.  For fixed time index
$j \leq N$, a temporal cell is mapped onto at least one full spatial
cell.  Thus, we can use a simple backward lookup as follows.  Let
\begin{equation}
  \Jc(i;j) = \lfloor i \tfrac{j}N \rfloor
  \label{e.ji}
\end{equation}
be the time index in the past that corresponds best to spatial index
$i$.  Then we set
\begin{equation}
  p_i^j = p_N^{\Jc(i;j)} \,.
\end{equation}

For a fixed time index $j>N$, we do a forward mapping, i.e., we define
the inverse function to \eqref{e.ji},
\begin{equation}
  \Ic(k;j) = \lceil k \tfrac{N}j \rceil \,,
\label{e.ij}
\end{equation}
which represents the spatial index that the cell with past time index
$k$ and spatial index $N$ has moved to, and set
\begin{equation}
  \label{f.ij}
  p_i^j = \frac{N}j \sum_{\Ic(k;j)=i} p_N^k \,.
\end{equation}
Note that this expression can yield values for $p_i^j$ outside of the
unit interval, which is not a problem as the integral over the entire
interval is represented correctly.  To implement this efficiently in
code, we keep a running sum
\begin{gather}
  P_j = \sum_{k = 0}^j p_N^k
\end{gather}
which can be updated incrementally and write
\begin{gather}
  p_i^j = \frac{N}j \, (P_{\Jc(i+1;j)} - P_{\Jc(i;j)}) \,.
\end{gather}
This expression is equivalent to \eqref{f.ij}.

Finally, we need to determine the precipitation function $p_i^j$ for
$i \geq N$.  For the simplified model, we set $p_N^j=1$ whenever
$w_N^j \geq u^* - \Phi_N$ and $p_N^j=0$ otherwise; $p_i^j=0$ for all
$i >N$.

For the full model, we need to probe the entire region where
$i \geq N$.  In this region, whenever $w_i^j > u^* - \Phi_i$, the
maximum principle for the continuum problem, see
\cite{DarbenasHO:2018:LongTA}, implies that $u$ exceeds the
precipitation threshold on some curve contained in the region
$\{s \leq s_j\}$ which connects the point $(\eta_i,s_j)$ with the line
$\eta=\alpha$.  This implies that $p_k^j = 1$ for all
$N \leq k \leq i$.  Consequently, we only need to track the largest
index $I^j$ where precipitation takes place and set $p_k^j=1$ for
$k = N, \dots, I^j$.  To do so, observe that precipitation takes place
either when $u$ exceeds the threshold, or when a cell lies on a
characteristic curve where precipitation has taken place at the
previous time step.  This leads to the the expression
\begin{equation}
  I^j
  = \max \bigl\{
           \max \{ k \colon w_k^j > u^* - \Phi_k \},
           \lfloor I^{j-1} \, (j-1)/j \rfloor
         \bigr\} \,,
\end{equation}
which finalizes the description of the numerical scheme for the full
model.

We remark that the scheme for the full model is almost equivalent to
the scheme used in \cite{DarbenasHO:2018:LongTA} which was formulated
in terms of the concentration $u$.  Here, to allow for switching
between the full and the simplified model, it was necessary to
formulate both schemes in terms of the difference variable
$w = u - \phi$.  Thus, the scheme given here differs from the scheme
in \cite{DarbenasHO:2018:LongTA} by terms that account for the
difference between $\phi$ and its finite difference approximation.

\section*{Acknowledgments}

We thank Rein van der Hout for introducing us to the fast reaction
limit of the Keller--Rubinow model, Arndt Scheel for interesting
discussions on Liesegang rings and for bringing reference
\cite{DuleyFM:2017:KellerRM} to our attention, and an anonymous
referee for helpful suggestions that improved the presentation of the
paper.  This work was funded through Deutsche Forschungsgemeinschaft
(DFG, German Research Foundation) grant OL 155/5-1.  Additional
funding was received via the Collaborative Research Center TRR 181
``Energy Transfers in Atmosphere and Ocean'', also funded by the DFG
under project number 274762653.

\bibliographystyle{siam}
\bibliography{liesegang}

\end{document}